\documentclass[12pt]{article}
\usepackage[centertags]{amsmath}
\RequirePackage{srcltx}
\usepackage{amsfonts}
\usepackage{amssymb}
\usepackage{tikz}

\usepackage{latexsym}
\usepackage{amsthm}
\usepackage{newlfont}
\usepackage{graphicx}
\usepackage{listings}
\usepackage{booktabs}
\usepackage{abstract}
\date{}

\newlength{\defbaselineskip}
\newcommand{\setlinespacing}[1]%
           {\setlength{\baselineskip}{#1 \defbaselineskip}}

\newcommand{\N}{{\mathbb{N}}}

\newcommand{\actaqed}{\hfill $\actabox$}
{\medskip\noindent \textit{Proof of #1. }}%
{\actaqed \medskip}

\def\D{{\mathcal D}}

\def\C{{\mathcal C}}
\def\cB{\mathcal B}
\def\cC{{\mathcal C}}
\def\cS{{\mathcal S}}

\def \Tr{\mathcal T}
\def \K{\mathcal K}
\def \P{\mathcal P}
\def \V{\mathcal V}
\def \U{\mathcal U}

\def \cX{\mathcal X}

\def \cM{\mathcal M}
\def\cR{\mathcal R}
\def\R{{\mathbb R}}
\def\Z{\mathbb Z}

\def \Td{{\mathbb T}^d}
\def \T{\mathbb T}
\def\bP{\mathbb P}
\def\bE{\mathbb E}

\def \bbE{\mathbb E}
\def \<{\langle}
\def\>{\rangle}
\def \L{\Lambda}
\def \La{\Lambda}
\def \ep{\epsilon}
\def \e{\varepsilon}
\def \de{\delta}

\def\al{\alpha}

\def\la{\lambda}

\def \sp{\operatorname{span}}

\def \meas{\operatorname{meas}}

\def\ba{\mathbf a}
\def\bb{\mathbf b}
\def\bx{\mathbf x}
\def\by{\mathbf y}
\def\bz{\mathbf z}
\def\bk{\mathbf k}

\def\bw{\mathbf w}

\def\bs{\mathbf s}
\def\bN{\mathbf N}
\def\bW{\mathbf W}

\newtheorem{Theorem}{Theorem}[section]
\newtheorem{Lemma}{Lemma}[section]
\newtheorem{Definition}{Definition}[section]
\newtheorem{Proposition}{Proposition}[section]
\newtheorem{Remark}{Remark}[section]

\newtheorem{Corollary}{Corollary}[section]
\numberwithin{equation}{section}

\newcommand{\be}{\begin{equation}}
\newcommand{\ee}{\end{equation}}

\usepackage{enumerate}

\def\Br{\Bigr}
\def\Bl{\Bigl}
\def\f{\frac}
\def\bl{\bigl}
\def\br{\bigr}
\def\b{\beta}

\def\conv{\operatorname{conv}}

\begin{document}

\title{
 Integral norm discretization and related problems
\footnote{
The first named author's research was partially supported by NSERC of Canada Discovery Grant RGPIN 04702-15.
The second named author's research was partially supported by NSERC of Canada Discovery Grant RGPIN 04863-15.
The third named author's research was supported by the Russian Federation Government Grant No. 14.W03.31.0031.
The fourth  named
author's research was partially supported by
 MTM 2017-87409-P,  2017 SGR 358, and
 the CERCA Programme of the Generalitat de Catalunya.
}}

\author{F. Dai, \, A. Prymak, \, V.N. Temlyakov, \, and  \, S. Tikhonov}

\newcommand{\Addresses}{{
  \bigskip
  \footnotesize

  F.~Dai, \textsc{ Department of Mathematical and Statistical Sciences\\
University of Alberta\\ Edmonton, Alberta T6G 2G1, Canada\\
E-mail:} \texttt{fdai@ualberta.ca }

  \medskip
 A.~Prymak, \textsc{ Department of Mathematics\\
University of Manitoba\\ Winnipeg, MB, R3T 2N2, Canada
  \\
E-mail:} \texttt{Andriy.Prymak@umanitoba.ca }

    \medskip
  V.N. Temlyakov, \textsc{University of South Carolina,\\ Steklov Institute of Mathematics,\\ and Lomonosov Moscow State University
  \\
E-mail:} \texttt{temlyak@math.sc.edu}

  \medskip

  S.~Tikhonov, \textsc{Centre de Recerca Matem\`{a}tica\\
Campus de Bellaterra, Edifici C
08193 Bellaterra (Barcelona), Spain;\\
ICREA, Pg. Llu\'{i}s Companys 23, 08010 Barcelona, Spain,\\
 and Universitat Aut\`{o}noma de Barcelona\\
E-mail:} \texttt{stikhonov@crm.cat}

}}

\maketitle

\bigskip
\begin{abstract}
 The problem of replacing an integral norm with respect to a given probability measure by the corresponding integral norm with respect to a discrete measure is discussed in the paper.
 The above problem is studied for elements of finite dimensional spaces.
 { Also, discretization of the uniform norm of functions from a given finite dimensional subspace of continuous functions is studied.  We pay special attention to the case of the multivariate trigonometric polynomials with frequencies from a finite set with fixed cardinality. }  Both new results and a survey of known results are presented.
\end{abstract}

\tableofcontents

\section{Introduction}
\label{Intr}

We study discretization of $L_q$ norms of functions from finite dimensional subspaces. In the case of $1\le q<\infty$ this problem can be formulated as a problem
of numerical integration. Let us return to the question of discretization of $L_q$ norm after a general discussion of the numerical integration problem.

Numerical integration seeks good ways of approximating an integral
$$
\int_\Omega f(\bx)d\mu(\bx)
$$
by an expression of the form
\be\label{1.1ni}
\La_m(f,\xi) :=\sum_{j=1}^m\la_jf(\xi^j),\quad \xi=(\xi^1,\dots,\xi^m),\quad \xi^j \in \Omega,\quad j=1,\dots,m.
\ee
It is clear that we must assume that $f$ is integrable and defined at the points
 $\xi^1,\dots,\xi^m$. Expression~\eqref{1.1ni} is called a {\it cubature formula} $(\xi,\La)$ (if $\Omega \subset \R^d$, $d\ge 2$) or a {\it quadrature formula} $(\xi,\La)$ (if $\Omega \subset \R$) with nodes $\xi =(\xi^1,\dots,\xi^m)$ and weights $\La:=(\la_1,\dots,\la_m)\in\R^m$. We do not impose any {\it a priori} restrictions on nodes and weights. Some nodes may coincide and both positive and negative weights are allowed.

  Some classes of cubature formulas are of special interest. For instance, the Quasi-Monte Carlo cubature formulas, which have equal weights $1/m$, are important in applications. We use a special notation for these cubature formulas
 $$
 Q_m(f,\xi) :=\frac{1}{m}\sum_{j=1}^mf(\xi^j).
 $$
Other examples include positive weights and weights satisfying stability constraint $\sum_{j=1}^m|\lambda_j|<const$.


Typically, one is interested in {\it good cubature formulas} for a given function class. The term {\it good} can be understood in different ways. Cubature formulas providing
exact numerical integration for functions from a given class can be considered ``best''. If a cubature formula is not exact on a given class then we need to introduce a concept of error.  Following the standard approach, for a function class $\bW$ we introduce the concept of error of the cubature formula $\La_m(\cdot,\xi)$ by
\[
\La_m(\bW,\xi):= \sup_{f\in \bW} \left|\int_\Omega fd\mu -\La_m(f,\xi)\right|.
\]
The quantity $\La_m(\bW,\xi)$ is a classical characteristic of the quality of a given cubature formula $\La_m(\cdot,\xi)$. This setting is called {\it the worst case setting} in
the Information Based Complexity (see, e.g., \cite{Wo}). Notice that the above error characteristic provides
an absolute error independent of an individual function from the class.

{ Recently, in a number of papers (see \cite{VT158}, \cite{VT159}, \cite{VT160}) a systematic study of the problem of discretization of the $L_q$ norms of elements of finite dimensional subspaces has begun. The first results in this direction were obtained by Marcinkiewicz and
by Marcinkiewicz-Zygmund (see \cite{Z}) for discretization of the $L_q$ norms of the univariate trigonometric polynomials in 1930s. This is why we call discretization results of this kind the Marcinkiewicz-type theorems. We discuss here the way of discretization which uses function values at a fixed finite set of points. Therefore, this way can also be called {\it sampling discretization}. }
 
 We discuss this problem in a rather general setting. Let $\Omega$ be a compact subset of $\R^d$ and $\mu$ be a probability measure on $\Omega$. We consider the space $L_q(\Omega)=L_q(\Omega,\mu)$, $1\le q< \infty$, of functions satisfying
$$
\|f\|_q := \left(\int_\Omega |f|^qd\mu\right)^{1/q} <\infty.
$$
In the case $q=\infty$ we define $L_\infty(\Omega)=\mathcal C(\Omega)$ as the space of continuous functions  on $\Omega$ with
$$
\|f\|_\infty := \max_{\bx\in\Omega} |f(\bx)|.
$$
In a special case when $\Omega$ is a discrete set $\Omega_M=\{\bx^j\}_{j=1}^M$ of distinct points $\bx^j$, we consider the measure 
$\mu$ such that $\mu(\bx^j)=1/M$, $j=1,\dots,M$.

In the Marcinkiewicz-type discretization problems we study the numerical integration
problem for the class \[\bW:= X_N^q:=\{f\in L_q(\Omega)\cap X_N: \|f\|_q\le 1\},\] where $X_N$ is a finite dimensional subspace of $L_q(\Omega)$, $1\le q<\infty$. An important new feature of our approach is the measurement of the error -- we study the {\it relative} error of numerical integration. Let us now formulate explicitly the main problems of our interest.

 {\bf Marcinkiewicz problem.} Let $\Omega$ be a compact subset of $\R^d$ with the probability measure $\mu$. We say that a linear subspace $X_N$ (usually $N$ stands for the dimension of $X_N$) of the $L_q(\Omega)$, $1\le q < \infty$, admits the Marcinkiewicz-type discretization theorem with parameters $m$ and $q$ if there exist a set $\{\xi^\nu \in \Omega: \nu=1,\dots,m\}$ and two positive constants $C_j(d,q)$, $j=1,2$, such that for any $f\in X_N$ we have
\be\label{1.1}
C_1(d,q)\|f\|_q^q \le \frac{1}{m} \sum_{\nu=1}^m |f(\xi^\nu)|^q \le C_2(d,q)\|f\|_q^q.
\ee
In the case $q=\infty$ (recall that we set $L_\infty(\Omega)=\mathcal C(\Omega)$) we 
ask for
\[
C_1(d)\|f\|_\infty \le \max_{1\le\nu\le m} |f(\xi^\nu)| \le  \|f\|_\infty.
\]
We will also use a brief way to express the above property: the $\cM(m,q)$ theorem holds for  a subspace $X_N$ or $X_N \in \cM(m,q)$.

{\bf Marcinkiewicz problem with weights.}  We say that a linear subspace $X_N$ of the $L_q(\Omega)$, $1\le q < \infty$, admits the weighted Marcinkiewicz-type discretization theorem with parameters $m$ and $q$ if there exist a set of nodes $\{\xi^\nu \in \Omega\}$, a set of weights $\{\la_\nu\}$, $\nu=1,\dots,m$, and two positive constants $C_j(d,q)$, $j=1,2$, such that for any $f\in X_N$ we have
\be\label{1.5}
C_1(d,q)\|f\|_q^q \le  \sum_{\nu=1}^m \la_\nu |f(\xi^\nu)|^q \le C_2(d,q)\|f\|_q^q.
\ee
Then we also say that the $\cM^w(m,q)$ theorem holds for  a subspace $X_N$ or $X_N \in \cM^w(m,q)$.
Obviously, $X_N\in \cM(m,q)$ implies that $X_N\in \cM^w(m,q)$.

{\bf Marcinkiewicz problem with $\e$.} We write $X_N\in \cM(m,q,\e)$ if (\ref{1.1}) holds with $C_1(d,q)=1-\e$ and $C_2(d,q)=1+\e$.  Respectively,
we write $X_N\in \cM^w(m,q,\e)$ if (\ref{1.5}) holds with $C_1(d,q)=1-\e$ and $C_2(d,q)=1+\e$.
We  also write $ X_N\in \cM^w_+(m,q,\e)$ if   $X_N\in \cM^w(m,q,\e)$ and (\ref{1.5}) holds with nonnegative weights $\lambda_\nu$.
We note that the most powerful results are for $\cM(m,q,0)$,
when the $L_q$ norm of $f\in X_N$ is discretized exactly by the formula with equal weights $1/m$. In case $X_N\in \cM(m,q,0)$ we say that $X_N$ admits {\it exact discretization} with parameters $m$ and $q$. In case $X_N\in \cM^w(m,q,0)$ we say that $X_N$ admits {\it exact weighted discretization} with parameters $m$ and $q$.

In the above formulations of the problems we only ask about existence of
either good $\{\xi^\nu\}$ or good $\{\xi^\nu,\la_\nu\}$. Certainly, it is important to
have either explicit constructions of good $\{\xi^\nu\}$ ($\{\xi^\nu,\la_\nu\}$) or
deterministic ways to construct good $\{\xi^\nu\}$ ($\{\xi^\nu,\la_\nu\}$). Thus, the
Marcinkiewicz-type problem can be split into the following four problems: under some assumptions on $X_N$

\begin{description}
  \item[(I)]
   Find a condition on $m$ for $X_N \in \cM(m,q)$;
  \item[(II)]
Find a condition on $m$ for $X_N \in \cM^w(m,q)$;
  \item[(III)]
 Find a condition on $m$ such that there exists a deterministic construction
of $\{\xi^\nu\}_{\nu=1}^m$ satisfying (\ref{1.1}) for all $f\in X_N$;

  \item[(IV)]
Find a condition on $m$ such that there exists a deterministic construction
of $\{\xi^\nu,\la_\nu\}_{\nu=1}^m$ satisfying (\ref{1.5}) for all $f\in X_N$.
   \end{description}
We note that the setting of the Marcinkiewicz-type problems is motivated by
applications. For instance, a typical approach to solving a continuous problem numerically -- the Galerkin method --
suggests searching for an approximate solution from a given finite dimensional subspace. A standard way to measure an error of approximation is an appropriate $L_q$ norm, $1\le q\le\infty$. Thus, the problem of   discretization of the $L_q$ norms of functions from a given finite dimensional subspace arises in a very natural way.

The paper contains both new results and a brief survey.
 Section \ref{survey} provides a survey of
known results on the Marcinkiewicz-type discretization. This section does not contain
new results.

 In Sections \ref{Ex} and \ref{constr} we present new results with brief discussions. These results are devoted
to exact weighted discretization. In particular, Theorems~\ref{gfT1} and~\ref{gfT2}
solve completely the problem of exact weighted discretization for general finite dimensional subspaces. Theorem~\ref{thm-4-1} provides a more general version of the Tchakaloff's theorem (e.g., see~\cite{P}) with a different proof.

Section \ref{hyper} is devoted to the problem of Marcinkiewicz-type discretization in $L_\infty$ on the subspace of trigonometric polynomials with frequencies from a hyperbolic cross. This problem is still open (see Open problem 5 in the last section).
 Here we present new results, which complement
a phenomenon discovered earlier (see the discussion in Section \ref{survey}).

{ In Section \ref{ungen} we present recent results from \cite{KT168} and \cite{KT169} on sampling discretization of the uniform norm of elements of finite dimensional subspaces. These results show that the Marcinkiewicz-type inequalities in the uniform norm are different from their counterparts in $L_1$ and $L_2$. }

In Section \ref{ud} we address the following important from the point of view of applications 
feature of discretization -- {\it universality} (see \cite{Tem16} and \cite{TBook}). Universality means that we want to build a discretization pair $(\{\xi^\nu\},\{\lambda_\nu\})$ which is good for each subspace from a given collection instead of being good only for a single given subspace.  We give there (see Subsection 6.1) a detailed
survey of known results on universal discretization. Also, we present new results (see Subsection 6.2) on universal discretization.

 In Section \ref{OP} we present some open problems.

\section{A brief survey}
\label{survey}

\subsection{Trigonometric polynomials}\label{sec2.1}
In this subsection we deal with the $2\pi$-periodic case of $d$-variate functions. In this case $\Omega =\T^d$ and $\mu$ is a normalized Lebesgue measure on $\T^d$.
We discuss discretization theorems of Marcinkiewicz-type  for subspaces of the
trigonometric polynomials. By $Q$ we denote a finite subset of $\Z^d$, and $|Q|$ stands for the number of elements in $Q$. Let
$$
\Tr(Q):= \{f: f=\sum_{\bk\in Q}c_\bk e^{i(\bk,\bx)},\  \  c_{\bk}\in\mathbb{C}\}.
$$
Let us start with the well-known results related to the Marcinkiewicz-type discretization theorems for the trigonometric polynomials.
We first consider the case $Q=\Pi(\bN):=[-N_1,N_1]\times \cdots \times [-N_d,N_d]$, $N_j \in \N$ or $N_j=0$, $j=1,\dots,d$, $\bN=(N_1,\dots,N_d)$.
We set
\begin{align*}
P(\mathbf N) := \Bigl\{\mathbf n = (n_1 ,\dots,n_d)\in\Z^d:\
0\le n_j\le 2N_j  ,\ j = 1,\dots,d \Bigr\},
\end{align*}
and
$$
\bx^{\mathbf n}:=\left(\frac{2\pi n_1}{2N_1+1},\dots,\frac{2\pi n_d}
{2N_d+1}\right),\qquad \mathbf n\in P(\mathbf N) .
$$
For any $t\in \Tr(\Pi(\mathbf N))$, one  has
$$
\|t\|_2^2  =\vartheta(\mathbf N)^{-1}\sum_{\mathbf n\in P(\mathbf N)}
\bigl|t(\bx^{\mathbf n})\bigr|^2,
$$
where $\vartheta(\mathbf N) := \prod_{j=1}^d (2N_j  + 1)=\dim\Tr(\Pi(\bN))$.
In particular, this implies that for any $\bN$ one has
\be\label{1.3}
\Tr(\Pi(\bN)) \in \cM(\vartheta(\bN),2,0).
\ee
In the case $1<q<\infty$, 
 the
well-known Marcinkiewicz discretization theorem (for $d=1$) is given as follows (see \cite{Z}, Ch.10, \S7 and \cite{TBook}, Ch.1, Section~2): for $t\in \mathcal{T}(\Pi(\bN))$,
$$
C_1(d,q)\|t\|_q^q  \le\vartheta(\mathbf N)^{-1}\sum_{\mathbf n\in P(\mathbf N)}
\bigl|t(\bx^{\mathbf n})\bigr|^q  \le C_2(d,q)\|t\|_q^q,\quad 1<q<\infty.
$$
This yields the following extension of~\eqref{1.3}:
\[
\Tr(\Pi(\bN)) \in \cM(\vartheta(\bN),q),\quad 1<q<\infty.
\]
For $q=1$ or $q=\infty$, one needs some adjustments. Let
\begin{align*}
P'(\mathbf N) := \Bigl\{\mathbf n &= (n_1,\dots,n_d)\in\Z^d:\
1\le n_j\le 4N_j  ,\ j = 1,\dots,d \Bigr\}
\end{align*}
and
$$
\bx(\mathbf n) :=\left (\frac{\pi n_1}{2N_1} ,\dots,\frac{\pi n_d}{2N_d}
\right)   ,\qquad \mathbf n\in P'(\mathbf N)  .
$$
If $N_j  = 0$, we let $x_j (\mathbf n) = 0$. Set ${\overline N} := \max (N,1)$ and $\nu(\bN) := \prod_{j=1}^d {\overline N_j}$.
 Therefore, the following Marcinkiewicz-type discretization theorem 
$$
C_1(d,q)\|t\|_q^q  \le\nu(4\mathbf N)^{-1}\sum_{\mathbf n\in P'(\mathbf N)}
\bigl|t(\bx({\mathbf n}))\bigr|^q  \le C_2(d,q)\|t\|_q^q,\quad 1\le q\le \infty,
$$
implies that
\[
\Tr(\Pi(\bN)) \in \cM(\nu(4\bN),q),\quad 1\le q\le \infty.
\]
We note that $\nu(4\bN) \le C(d) \dim \Tr(\Pi(\bN))$.

Let us now discuss
  the Marcinkiewicz-type discretization theorems for the hyperbolic cross trigonometric polynomials (see~\cite{DTU} for a recent survey covering a variety of topics related to the hyperbolic cross approximation). For $\bs\in\Z^d_+$
  we define
$$
\rho (\bs) := \{\bk \in \Z^d : [2^{s_j-1}] \le |k_j| < 2^{s_j}, \quad j=1,\dots,d\}
$$
where $[x]$ denotes the integer part of $x$.
By $Q_n$ denote
the step hyperbolic cross, i.e.,
$$
Q_n := \bigcup_{\bs:\|\bs\|_1\le n} \rho(\bs).
$$
Then  the corresponding set of the hyperbolic cross polynomials is given by
$$
\Tr(Q_n) := \left\{f: f=\sum_{\bk\in Q_n} c_\bk e^{i(\bk,\bx)},\  \  c_\bk\in\mathbb{C}\right\}.
$$
The problem on
obtaining the sharp Marcinkiewicz-type discretization theorems for the hyperbolic cross trigonometric polynomials is not solved yet.
To the best of our knowledge, no sharp results on the growth of $m$ as a function on $n$ for the relation $\Tr(Q_n) \in \cM(m,q)$ to hold  for  $1\le q\le \infty$, $q\neq 2$, are known.
Since $Q_n \subset \Pi(2^n,\dots,2^n)$, from the above mentioned results we have
$$
\Tr(Q_n)\in \cM(m,q),\quad \text{provided} \quad m\ge C(d)2^{dn},\quad 1\le q\le \infty,
$$
with large enough $C(d)$.
It seems that the first nontrivial result related to this problem  was derived in \cite{VT27}, where the set of points $\{\xi^\nu\}_{\nu=1}^p$ with $p\ll 2^{2n}n^{d-1}$ such that for all $t\in \Tr(Q_n)$ the inequality
$$
\|t\|_2^2 \le   \frac{1}{p} \sum_{\nu=1}^p |t(\xi^\nu)|^2
$$
holds was constructed. Later on,  a very nontrivial surprising negative result was obtained in the case of $q=\infty$ (see \cite{KT3, KT4, KaTe03} and Section \ref{hyper} below). It was proved that in order to have
$\Tr(Q_n)\in\cM(m,\infty)$ it is necessary that
 $m\gg |Q_n|^{1+c}$ with absolute constant $c>0$.

Moreover, it is worth mentioning that some deep general results on submatrices of orthogonal matrices imply
  important  Marcinkiewicz-type discretization theorems for $q=2$.
 For example, Rudelson's theorem \cite{Rud} yields  the following result
\[
\Tr(Q_n)\in \cM(m,2),\quad \text{provided} \quad m\ge C(d)|Q_n|n
\]
with large enough $C(d)$; see also Subsection \ref{subsection General subspaces 2}.

Let now discuss a recent breakthrough result by J. Batson, D.A. Spielman, and N. Srivastava \cite{BSS}, which will be written  using our notations.
\begin{Theorem}[\cite{BSS}] \label{thm:BSS}
	Let  $\Omega_M=\{x^j\}_{j=1}^M$ be a discrete set with the probability measure $\mu(x^j)=1/M$, $j=1,\dots,M$.
Let also
$\{u_i(x)\}_{i=1}^N$ be a real system of functions on $\Omega_M$.
Then for any
number $b>1$ there exist a set of weights $w_j\ge 0$ such that $|\{j: w_j\neq 0\}| \le bN$ so that for any $f\in Y_N:= \sp\{u_1,\dots,u_N\}$ we have 
\be\label{C2'}
\|f\|_{L_2(\Omega_M)}^2 \le \sum_{j=1}^M w_jf(x^j)^2 \le \frac{b+1+2\sqrt{b}}{b+1-2\sqrt{b}}\|f\|_{L_2(\Omega_M)}^2.
\ee
\end{Theorem}
As a particular case, we obtain the weighted version of
  the $L_2$ Marcinkiewicz-type discretization theorem, that is,
     (\ref{1.1}) holds  for the above $X_N$ with $m\ge cN$
     with the general  weights $w_j$ instead of weights $1/m$.

The next theorem was derived in~\cite{VT158} from the
 recent paper by  S.~Nitzan, A.~Olevskii, and A.~Ulanovskii~\cite{NOU}, which in turn is based on the paper of A.~Marcus, D.A.~Spielman, and N.~Srivastava~\cite{MSS}.

\begin{Theorem}[{\cite[Theorem~1.1]{VT158}}] \label{NOUth}There are three positive absolute constants $C_1$, $C_2$, and $C_3$ with the following properties: For any $d\in \N$ and any $Q\subset \Z^d$   there exists a set of  $m \le C_1|Q| $ points $\xi^j\in \T^d$, $j=1,\dots,m$ such that for any $f\in \Tr(Q)$
	we have
	$$
	C_2\|f\|_2^2 \le \frac{1}{m}\sum_{j=1}^m |f(\xi^j)|^2 \le C_3\|f\|_2^2.
	$$
\end{Theorem}
In other words, Theorem \ref{NOUth} provides a solution of the Marcinkiewicz-type discretization theorem for the $\Tr(Q)$ in the $L_2$ case for any $Q$. For more details regarding
  the $L_2$ case see Subsection \ref{subsection General subspaces 2}, the paper
\cite{VT159}, and Kashin's paper \cite{Ka}, where the author discusses  a
recent spectacular  progress in the area of submatrices of orthogonal matrices.

We formulate some other results of \cite{VT158}. It was proved there that for $d=2$
$$
\Tr(Q_n) \in \cM(m,1),\quad \text{provided} \quad m\ge C |Q_n|n^{7/2}
$$
with large enough $C$,
and for $d\ge 3$
$$
\Tr(Q_n) \in \cM(m,1),\quad \text{provided} \quad m\ge C(d) |Q_n|n^{d/2+3}
$$
with large enough $C(d)$.
The above result was improved in \cite{VT159} to the following one. For any $Q\subset \Pi(\bN)$ with $\bN=(2^n,\dots,2^n)$ we have
$$
\Tr(Q) \in \cM(m,1),\quad \text{provided} \quad m\ge C(d) |Q_n|n^{7/2}
$$
with large enough $C(d)$.
This gives rise to the following intriguing open problem (see also open problem 5 in the last section):
does the relation $\Tr(Q_n) \in \cM(m,1)$ hold with $ m\asymp |Q_n|$?

We note that the results of \cite{VT158} and \cite{VT159} mentioned above were derived using probabilistic technique.
In more detail, the following ingredients were used: a variant of the Bernstein concentration measure inequality
from \cite{BLM}, the chaining technique from \cite{KoTe} (see also \cite{Tbook2}, Ch.4), and the recently obtained \cite{VT156} bounds of the entropy numbers. It is worth mentioning that the application of chaining technique was initiated by A.N. Kolmogorov in the 30s of the last century.
 After that results
   of these type were established  in the study of the central limit theorem in probability theory (see, e.g., \cite{GZ}).
      See also \cite{Ta} for further results on the chaining technique.


Let us stress again that  the approach used in~\cite{VT158} is based on the probabilistic technique.
As a consequence, we derive the existence of good points for the Marcinkiewicz-type discretization theorems,
but no algorithm of construction of these points is offered.
We believe that the problem of
deterministic constructions of point sets, which give at least the same bounds for $m$ as the probabilistic approach does,
is of great importance.
 A deterministic construction based on number theoretical
considerations was suggested in \cite{VT158}.
Even though
this approach can be applied for quite general finite sets $Q\subset \Z^d$, it is restricted to the case $q=2$.

Let us discuss the case $q=2$ in more detail.
First, using the probabilistic technique,
one proves the Marcinkiewicz-type discretization theorem for $m\ge C |Q_n| $ with some large enough constant $C$ (see Theorem \ref{NOUth}).
Second,
 the deterministic Marcinkiewicz-type discretization theorem in $L_2$ holds (see \cite{VT158}) for $m\ge C(d) |Q_n|^2 $ with large enough constant $C(d)$
in the exact form with $C_1(d,2)=C_2(d,2)=1$ (see Sections \ref{Intr} and \ref{Ex}).
   Namely, the exact discretization theorem states that for a given set $Q$ we construct a set $\{\xi^\nu\}_{\nu=1}^m$ with $m\le C(d)|Q|^2$ such that for any $t\in \Tr(Q)$ we have
$$
\|t\|_2^2 = \frac{1}{m}\sum_{\nu=1}^m |t(\xi^\nu)|^2.
$$
Note that
the probabilistic approach
 requires bounds on the entropy numbers $\e_k(\Tr(Q)_q,L_\infty)$ of the unit $L_q$ balls of $\Tr(Q)$ in $L_\infty$, which is  a deep and demanding question
  by itself. To attack this problem, an approach using greedy approximation methods  has been recently  established in \cite{VT156}.

We discussed in \cite{KT168} and \cite{KT169} the following setting of the discretization problem of the uniform norm.
Let $S_m:=\{\xi^j\}_{j=1}^m \subset \Td$ be a finite set of points. Clearly,
$$
\|f\|_{L_\infty(S_m)} := \max_{1\le j\le m} |f(\xi^j)| \le \|f\|_\infty.
$$
We are interested in estimating the following quantities
$$
D(Q,m):=D(Q,m,d):= \inf_{S_m}\sup_{f\in\Tr(Q)}\frac{\|f\|_\infty}{\|f\|_{L_\infty(S_m)}},
$$
$$
D(N,m):=D(N,m,d):= \sup_{Q,|Q|=N} D(Q,m,d).
$$
Certainly, one should assume that $m\ge N$. Then the characteristic $D(Q,m)$ guarantees that there exists a set of $m$ points $S_m$ such that for any $f\in\Tr(Q)$ we have
$$
\|f\|_\infty\le D(Q,m)\|f\|_{L_\infty(S_m)}.
$$
In the case $d=1$ and $Q=[-n,n]$ classical Marcinkiewicz theorem (see \cite{VTbookMA}, p. 24)
gives for $m\ge 4n$ that $D([-n,n],4n)\le C$. Similar relation holds for $D([-n_1,n_1]\times\cdots\times[-n_d,n_d], (4n_1)\times\cdots\times(4n_d))$ (see \cite{VTbookMA}, p. 102).

It was proved in \cite{KT169} that for a pair $N$, $m$, such that $m\asymp N$ we have $D(N,m)\asymp N^{1/2}$. We formulate this result  as a theorem.

\begin{Theorem}[\cite{KT169}] \label{ITmain} For any constant $c\ge 1$ there exists a positive constant $C$ such that for any pair of parameters $N$, $m$, with $m\le cN$ we have
$$
D(N,m)\ge CN^{1/2}.
$$
Also, there are two positive absolute constants $c_1$ and $C_1$ with the following property: For any $d\in \N$ we have for $m\ge c_1N$
$$
D(N,m,d)\le C_1N^{1/2}.
$$
\end{Theorem}
The first part of Theorem \ref{ITmain} follows from Corollary \ref{BC1} (see (\ref{B6})) and the second part follows from Theorem \ref{CT2}.

It is interesting to compare Theorem \ref{ITmain}, which provides a result on discretization of the uniform norm, with the cited above known result -- Theorem \ref{NOUth} -- on discretization of the $L_2$ norm.

\subsection{General subspaces in $L_1$}

We begin with the definition of the entropy numbers.
Let $X$ be a Banach space and let $B_X$ denote the unit ball of $X$ with the center at $0$. Denote by $B_X(y,r)$ a ball with center $y$ and radius $r$, that is,  $B_X(y,r)=\{x\in X:\|x-y\|\le r\}$. For a compact set $A$ and a positive number $\e$ we define the covering number $N_\e(A)$
as follows
$$
N_\e(A) := N_\e(A,X)
:=\min \{n : \exists y^1,\dots,y^n\in A,\   A\subseteq \cup_{j=1}^n B_X(y^j,\e)\}.
$$
It is convenient to consider along with the entropy $H_\e(A,X):= \log_2 N_\e(A,X)$  the entropy numbers $\e_k(A,X)$:
$$
\e_k(A,X)  :=\inf \{\e : \exists y^1,\dots ,y^{2^k} \in A ,\   A \subseteq \cup_{j=1}
^{2^k} B_X(y^j,\e)\}.
$$
In our definition of $N_\e(A)$ and $\e_k(A,X)$ we require $y^j\in A$. In a standard definition of $N_\e(A)$ and $\e_k(A,X)$ this restriction is not imposed.
However, it is well known (see \cite{Tbook2}, p.208) that these characteristics may differ at most by a factor $2$.
The following general conditional result
has been recently  obtained in \cite{VT159}.

\begin{Theorem}[\cite{VT159}] \label{T4.10} Suppose that the $L_1$ unit ball $X^1_N:=\{f\in X_N: \|f\|_1\le 1\}$ of a subspace $X_N$ satisfies the condition $(B\ge 1)$
	$$
	\e_k(X^1_N,L_\infty) \le  B\left\{\begin{array}{ll}  N/k, &\quad k\le N,\\
	2^{-k/N},&\quad k\ge N.\end{array} \right.
	$$
	Then for large enough absolute constant $C$ there exists a set of  $$m \le CNB(\log_2(2N\log_2(8B)))^2$$ points $\xi^j\in \Omega$, $j=1,\dots,m$,   such that for any $f\in X_N$
	we have
	$$
	\frac{1}{2}\|f\|_1 \le \frac{1}{m}\sum_{j=1}^m |f(\xi^j)| \le \frac{3}{2}\|f\|_1.
	$$
\end{Theorem}
In particular,
this result shows that the investigation  of the entropy numbers $\e_k(X^1_N,L_\infty)$ plays a crucial
 role to prove  the Marcinkiewicz discretization theorems in $L_1$.

The study
of the entropy numbers is a  highly nontrivial and intrinsically interesting subject. Let us show  this for 
 trigonometric polynomials.
On the one hand, it is known \cite{VT156} that in the case $d=2$
one has
\be\label{6.1}
\e_k(\Tr( Q_n)_1,L_\infty)\ll  n^{1/2} \left\{\begin{array}{ll} (| Q_n|/k) \log (4| Q_n|/k), &\quad k\le 2| Q_n|,\\
	2^{-k/(2| Q_n|)},&\quad k\ge 2| Q_n|,\end{array}
\right.
\ee
where
 $\Tr( Q_n)_1=\{f\in \Tr( Q_n) : \|f\|_1\le 1\}$.
 The proof of the  estimate (\ref{6.1}) relies  on a version  of the Small Ball Inequality for the trigonometric system obtained for the wavelet type system (see \cite{VT156}). This proof is strongly based on the two-dimensional structure and its extension for higher dimensional case is problematic. 
 On the other hand, by
the trivial estimate $\log (4| Q_n|/k) \ll n$,
(\ref{6.1}) yields the following inequality
\be\label{6.2}
\e_k(\Tr( Q_n)_1,L_\infty)\ll  n^{3/2} \left\{\begin{array}{ll} | Q_n|/k , &\quad k\le 2| Q_n|,\\
	2^{-k/(2| Q_n|)},&\quad k\ge 2| Q_n|.\end{array} \right.
\ee
Even though to obtain new upper bounds of the entropy numbers of smoothness classes
the latter inequality  is less applicable  than estimate (\ref{6.1}),
both estimates (\ref{6.1}) and (\ref{6.2}), applied  to the Marcinkiewicz-type discretization theorems,
  give the same bounds on the number of nodes $m\ll |Q_n|n^{7/2}$.

As  was mentioned above,
an extension of (\ref{6.1}) to the case
$d>2$ is not established. A somewhat straightforward technique given
in \cite{VT158} allows us to claim that  for all $d$
\[
\e_k(\Tr( Q_n)_1,L_\infty)\ll  n^{d/2} \left\{\begin{array}{ll} (| Q_n|/k) \log (4| Q_n|/k), &\quad k\le 2| Q_n|,\\
	2^{-k/(2| Q_n|)},&\quad k\ge 2| Q_n|.\end{array} \right.
\]
This can be used to derive the Marcinkiewicz inequality (\ref{1.1}) in $L_1$ (see \cite{VT158}).
We stress that in the paper \cite{VT159} the proof of (\ref{6.2}) is given for all $d$ and for  general sets $\Tr(Q)_1$ instead of $\Tr(Q_n)_1$.

A very interesting open question is to investigate, even in the special case of the hyperbolic cross polynomials $\Tr(Q_n)$, if the relation $\Tr(Q_n) \in \cM(m,1)$ with $m\asymp |Q_n|$ is valid.
From the results of \cite{VT158} and \cite{VT159},   the above relation holds with $m \gg |Q_n|n^{7/2}$.
The extra factor $n^{7/2}$ appears as a result of applying (\ref{6.2}), which contributed $n^{3/2}$,
and of applying the chaining technique, which contributed $n^2$.

 Let $X_N=\sp(u_1,\dots,u_N)$ be a real subspace of $L_1(\Omega)$.
Let us impose  several assumptions  on the system $\{u_i\}_{i=1}^N$ of real functions, which are needed to state the
 discretization result in the case $q=1$ (\cite{VT159}).

{\bf A.} There exist $\alpha>0$, $\beta$, and $K_1$ such that for all $i\in\{1,\dots,N\}$ we have
\[
|u_i(\bx)-u_i(\by)| \le K_1N^\beta\|\bx-\by\|_\infty^\alpha,\quad \bx,\by \in \Omega.
\]

{\bf B.} There exists a constant $K_2$ such that $\|u_i\|_\infty^2 \le K_2$, $i=1,\dots,N$.

{\bf C.} Denote $X_N:= \sp(u_1,\dots,u_N)$. There exist two constants $K_3$ and $K_4$ such that the following Nikol'skii-type inequality holds for all $f\in X_N$
\[
\|f\|_\infty \le K_3N^{K_4/p}\|f\|_p,\quad p\in [2,\infty).
\]
Now we are in a position to formulate the main result of \cite{VT159}.

\begin{Theorem}[\cite{VT159}] \label{T6.1} Suppose that a real orthonormal system $\{u_i\}_{i=1}^N$ satisfies conditions {\bf A}, {\bf B}, and {\bf C}. Then for large enough $C_1=C(d,K_1,K_2,K_3,K_4,\Omega,\alpha,\beta)$ there exists a set of $m \le C_1N(\log N)^{7/2}$ points $\xi^j\in \Omega$, $j=1,\dots,m$,   such that for any $f\in X_N$
	we have
	$$
	\frac{1}{2}\|f\|_1 \le \frac{1}{m}\sum_{j=1}^m |f(\xi^j)| \le \frac{3}{2}\|f\|_1.
	$$
\end{Theorem}

\subsection{General subspaces in $L_2$}\label{subsection General subspaces 2}

In this subsection we consider some known results related to the discretization theorems and, in particular, we discuss applications of the recent results on random matrices to derive the Marcinkiewicz-type theorem in $L_2$.
We start with an important result on submatrices of an orthogonal matrix obtained by M. Rudelson. 
Let us formulate it in our notations. 

\begin{Theorem}[\cite{Rud}]
	Let $\Omega_M=\{x^j\}_{j=1}^M$ be a discrete set with the probability measure $\mu(x^j)=1/M$, $j=1,\dots,M$. Let also
$\{u_i(x)\}_{i=1}^N$ be a real orthonormal  system on $\Omega_M$ satisfying the following condition: for all $j$
\be\label{5.1}
\sum_{i=1}^Nu_i(x^j)^2 \le Nt^2
\ee
with some $t\ge 1$.
Then for every $\ep>0$ there exists a set $J\subset \{1,\dots,M\}$ of indices with cardinality
\be\label{5.1a}
m:=|J| \le C\frac{t^2}{\ep^2}N\log\frac{Nt^2}{\ep^2}
\ee
such that for any $f=\sum_{i=1}^N c_iu_i$ we have
$$
(1-\ep)\|f\|_{L_2(\Omega_M)}^2 \le \frac{1}{m} \sum_{j\in J} f(x^j)^2 \le (1+\ep)\|f\|_{L_2(\Omega_M)}^2.
$$
\end{Theorem}
As a corollary,
this result  yields that if an orthonormal system $\{u_i\}_{i=1}^N$ on $\Omega_M$ satisfies (\ref{5.1}), one has
$$
{\mathcal U}_N := \sp(u_1,\dots,u_N) \in \cM(m,2)\quad \text{provided}\quad m\ge CN\log N
$$
with large enough $C$.

We remark that  condition (\ref{5.1}) is fulfilled if the system $\{u_i\}_{i=1}^N$ is uniformly bounded: $\|u_i\|_{L_\infty(\Omega_M)}\le t$, $i=1,\dots,N$.

To state the next result,
we need the following condition on
the system $\{u_j\}_{j=1}^N$, cf.~\eqref{5.1}.

{\bf Condition E.} There exists a constant $t$ such that
\be\label{ud5}
w(x):=\sum_{i=1}^N u_i(x)^2 \le Nt^2, \quad x\in\Omega.
\ee

\begin{Theorem}[\cite{VT159}] \label{T5.4} Let $\{u_i\}_{i=1}^N$ be a real  orthonormal system, satisfying condition {\bf E}.
Then for every $\ep>0$ there exists a set $\{\xi^j\}_{j=1}^m \subset \Omega$ with
	$$
	m  \le C\frac{t^2}{\ep^2}N\log N
	$$
	such that for any $f=\sum_{i=1}^N c_iu_i$ we have
	$$
	(1-\ep)\|f\|_2^2 \le \frac{1}{m} \sum_{j=1}^m f(\xi^j)^2 \le (1+\ep)\|f\|_2^2.
	$$
\end{Theorem}
Let us compare this theorem with the Rudelson result. First, Theorem~\ref{T5.4} establishes the Marcinkievicz-type discretization theorem for a general domain $\Omega$ instead of a discrete set $\Omega_M$.
Second, in  Theorem \ref{T5.4}  we have the  $\log N$ term in place of  $\log\frac{Nt^2}{\ep^2}$ in (\ref{5.1a}).


In its turn, the proof of
Theorem \ref{T5.4} rests  on the following result  on random matrices.

\begin{Theorem}[{\cite[Theorem 1.1]{Tro12}}] \label{T5.3} Consider a finite sequence $\{T_k\}_{k=1}^m$ of independent, random, self-adjoint matrices with dimension $N$. Assume that
	each random matrix is semi-positive and satisfies
	$$
	\lambda_{\max}(T_k) \le R\quad  \text{almost surely}.
	$$
	Define
	$$
	s_{\min} := \lambda_{\min}\left(\sum_{k=1}^m \bE(T_k)\right) \quad \text{and}\quad
	s_{\max} := \lambda_{\max}\left(\sum_{k=1}^m \bE(T_k)\right).
	$$
	Then
	$$
	\bP\left\{\lambda_{\min}\left(\sum_{k=1}^m T_k\right) \le (1-\eta)s_{\min}\right\} \le
	N\left(\frac{e^{-\eta}}{(1-\eta)^{1-\eta}}\right)^{s_{\min}/R}
	$$
	for $\eta\in[0,1)$ and
	$$
	\bP\left\{\lambda_{\max}\left(\sum_{k=1}^m T_k\right) \ge (1+\eta)s_{\max}\right\} \le
	N\left(\frac{e^{\eta}}{(1+\eta)^{1+\eta}}\right)^{s_{\max}/R},
	$$
for $\eta\ge 0$.
\end{Theorem}

\subsection{General subspaces in $L_\infty$ and $L_2$}
\label{gensub}

We now demonstrate how the above Theorem \ref{NOUth}, which, basically, solves the problem of the Marcinkiewicz-type discretization for the $\Tr(Q)$ in the $L_2$ case, was used in \cite{KT169}  in discretization of the uniform norm. 

\begin{Theorem}[\cite{KT169}] \label{CT2} There are two positive absolute constants $C_1$ and $C_4$ with the following properties: For any $d\in \N$ and any $Q\subset \Z^d$   there exists a set $S_m$ of  $m \le C_1|Q| $ points $\xi^j\in \T^d$, $j=1,\dots,m$, such that for any $f\in \Tr(Q)$
we have
$$
\|f\|_\infty \le C_4|Q|^{1/2}\|f\|_{L_\infty(S_m)}.
$$
\end{Theorem}
\begin{proof} We use the set of points provided by Theorem \ref{NOUth}. Then $m\le C_1|Q|$ and for any $f\in\Tr(Q)$ we have
\begin{align*}
\|f\|_\infty & \le |Q|^{1/2}\|f\|_2 \le |Q|^{1/2} C_2^{-1/2} \left(\frac{1}{m}\sum_{j=1}^m |f(\xi^j)|^2\right)^{1/2} \\ & \le |Q|^{1/2} C_2^{-1/2} \|f\|_{L_\infty(S_m)}.
\end{align*}
\end{proof}
We now present some results for more general subspaces than $\Tr(Q)$, which we discussed above. 

\begin{Theorem}[\cite{VT158}] \label{CT3} Let  $\Omega_M=\{x^j\}_{j=1}^M$ be a discrete set with the probability measure $\mu(x^j)=1/M$, $j=1,\dots,M$. Assume that
$\{u_i(x)\}_{i=1}^N$ is an orthonormal on $\Omega_M$ system (real or complex). Assume in addition that this system has the following property: for all $j=1,\dots, M$ we have
\be\label{C1}
\sum_{i=1}^N |u_i(x^j)|^2 = N.
\ee
Then there is an absolute constant $C_1$ such that there exists a subset $J\subset \{1,2,\dots,M\}$ with the property:  $m:=|J| \le C_1 N$ and
 for any $f\in Y_N:= \sp\{u_1,\dots,u_N\}$ we have
$$
C_2 \|f\|_{L_2(\Omega_M)}^2 \le \frac{1}{m}\sum_{j\in J} |f(x^j)|^2 \le C_3 \|f\|_{L_2(\Omega_M)}^2,
$$
where $C_2$ and $C_3$ are absolute positive constants.
\end{Theorem}

We note that assumption (\ref{C1}) implies the discrete Nikol'skii inequality for $f\in Y_N$
\be\label{C2a}
\|f\|_{L_\infty(\Omega_M)} \le N^{1/2}\|f\|_{L_2(\Omega_M)}.
\ee

In the same way as we derived above Theorem \ref{CT2} from Theorem \ref{NOUth} and the Nikol'skii inequality we derive the following Theorem \ref{CT4} from Theorem~\ref{CT3} and (\ref{C2a}).

\begin{Theorem}\label{CT4} Let  $\Omega_M=\{x^j\}_{j=1}^M$ be a discrete set with the probability measure $\mu(x^j)=1/M$, $j=1,\dots,M$. Assume that
$\{u_i(x)\}_{i=1}^N$ is an orthonormal on $\Omega_M$ system (real or complex). Assume in addition that this system has the following property: for all $j=1,\dots, M$ we have
\be\label{C1'}
\sum_{i=1}^N |u_i(x^j)|^2 = N.
\ee
Then there is an absolute constant $C_1$ such that there exists a subset $J\subset \{1,2,\dots,M\}$ with the property:  $m:=|J| \le C_1 N$ and
 for any $f\in Y_N:= \sp\{u_1,\dots,u_N\}$ we have
$$
 \|f\|_{L_\infty(\Omega_M)} \le C_4N^{1/2}\max_{j\in J} |f(x^j)|,
$$
where $C_4$ is an absolute positive constant.
\end{Theorem}

We now comment on a recent 
result by J. Batson, D.A. Spielman, and N. Srivastava \cite{BSS} stated above in Theorem~\ref{thm:BSS}.  Considering a new subspace $Y'_N:=\sp\{1,u_1,\dots,u_N\}$ and applying the above result we see that in the above result we can list one more property of weights $w_j$: $\sum_{j=1}^M w_j \le C(b)$. Therefore, in the same way as we proved Theorem \ref{CT2} we can prove the following Theorem \ref{CT5} (see \cite{KT169}).
\begin{Theorem}\label{CT5}
 Let  $\Omega_M=\{x^j\}_{j=1}^M$ be a discrete set with the probability measure $\mu(x^j)=1/M$, $j=1,\dots,M$. Assume that a real subspace $Y_N$ satisfies the Nikol'skii-type inequality: for any $f\in Y_N$
 $$
 \|f\|_{L_\infty(\Omega_M)} \le H(N)\|f\|_{L_2(\Omega_M)}.
 $$
 Then for any $a>1$ there exists a subset $J\subset \{1,2,\dots,M\}$ with the property:  $m:=|J| \le a N$ and
 for any $f\in Y_N:= \sp\{u_1,\dots,u_N\}$ we have
$$
 \|f\|_{L_\infty(\Omega_M)} \le C(a)H(N)\max_{j\in J} |f(x^j)|,
$$
where $C(a)$ is a  positive constant.

 \end{Theorem}

 An important feature of the above Theorems \ref{CT3} -- \ref{CT5} is that the domain is a discrete set $\Omega_M$. However, the statements of those theorems do not depend on $M$. This allows us to easily generalize some of those results to the case of general domain $\Omega$. We illustrate it on the example of generalization of Theorem \ref{CT5}. The way to do that is based on good approximation of $\|f\|_{L_2(\Omega)}$ and $\|f\|_\infty$ by $\|f\|_{L_2(\Omega_M)}$ and $\|f\|_{\Omega_M}$ respectively. We begin with the $L_2$ case.

 \begin{Proposition}\label{CP1} Let $Y_N:=\sp(u_1(\bx),\dots,u_N(\bx))$ with $\{u_i(\bx)\}_{i=1}^N$ being a real orthonormal on $\Omega$ with respect to a probability measure $\mu$ basis for $Y_N$. Assume that $\|u_i\|_4:=\|u_i\|_{L_4(\Omega,\mu)} <\infty$ for all $i=1,\dots,N$.   Then for any $\de>0$ there exists
 a set $\Omega_M=\{\bx^j\}_{j=1}^M$ such that for any $f\in Y_N$
 \be\label{C6'}
| \|f\|_{L_2(\Omega)}^2 - \|f\|_{L_2(\Omega_M)}^2| \le \de \|f\|_{L_2(\Omega)}^2
\ee
 where
 $$
 \|f\|_{L_2(\Omega_M)}^2 := \frac{1}{M}\sum_{j=1}^M |f(\bx^j)|^2.
 $$
 \end{Proposition}
 \begin{proof}
 Consider a real function $f\in L_2(\Omega):=L_2(\Omega,\mu)$ with respect to a probability measure $\mu$. Define $\Omega^M:= \Omega\times\cdots\times\Omega$ and $\mu^M:=\mu \times\cdots\times \mu$. For $\bx^j\in \Omega$ denote $\bz:= (\bx^1,\dots,\bx^M)\in \Omega^M$ and for $g\in L_1(\Omega^M, \mu^M)$
$$
\bbE(g):= \int_{\Omega^M}g(\bz)d\mu^M.
$$
Then it is well known from the study of the Monte Carlo integration method that we have for $g\in L_2(\Omega,\mu)$
\be\label{C2}
\bbE\left(\left(\int_\Omega g d\mu - \frac{1}{M}\sum_{j=1}^M g(\bx^j)\right)^2\right) \le  \|g\|_2^2/M.
\ee
 Denote $U:= \max_{1\le i\le N} \|u_i\|_4$. Then for $g=u_iu_j$, $1\le i,j\le N$ we find from (\ref{C2}) and the Markov inequality that for any $\e>0$ we have
\be\label{C3}
\mu^M\left\{\bz: \left(\int_\Omega g d\mu - \frac{1}{M}\sum_{j=1}^M g(\bx^j)\right)^2 \ge \e\right\} \le
\frac{U^4}{\e M}.
\ee
Therefore, for any $\e>0$ we can find big enough $M=M(\e,N)$ such that there exists a set
$\Omega_M=\{\bx^j\}_{j=1}^M$ such that for all $g$ of the form $g=u_iu_j$, $1\le i,j\le N$
we have
\be\label{C4}
\left|\int_\Omega g d\mu - \frac{1}{M}\sum_{j=1}^M g(\bx^j)\right| \le \e^{1/2}.
\ee
Consider $f= \sum_{i=1}^N b_iu_i$. Then $\|f\|_{L_2(\Omega)}^2=\sum_{i=1}^N b_i^2$.
Inequality (\ref{C4}) implies
\be\label{C5}
| \|f\|_{L_2(\Omega)}^2 - \|f\|_{L_2(\Omega_M)}^2| \le \e^{1/2} N \|f\|_{L_2(\Omega)}^2.
\ee
Now, for a $\de>0$ choosing $\e= \de^2 N^{-2}$ we obtain from (\ref{C5}) that
\be\label{C6}
| \|f\|_{L_2(\Omega)}^2 - \|f\|_{L_2(\Omega_M)}^2| \le \de \|f\|_{L_2(\Omega)}^2.
\ee
\end{proof}

Proposition \ref{CP1} and the above mentioned fundamental result (\ref{C2'}) imply the following discretization result.

\begin{Theorem}\label{CT5'} Let $Y_N:=\sp(u_1(\bx),\dots,u_N(\bx))$ with $\{u_i(\bx)\}_{i=1}^N$ being a real orthonormal on $\Omega$ with respect to a probability measure $\mu$ basis for $Y_N$. Assume that $\|u_i\|_4:=\|u_i\|_{L_4(\Omega,\mu)} <\infty$ for all $i=1,\dots,N$.
Then for any
number $a>1$ there exist a set of points $S_m=\{\xi^j\}_{j=1}^m$ and a set of positive weights $\{w_j\}_{j=1}^m$ with $m \le aN$ so that for any $f\in Y_N:= \sp\{u_1,\dots,u_N\}$ we have
\be\label{C2''}
\frac{1}{2}\|f\|_2^2 \le \sum_{j=1}^m w_jf(\bx^j)^2 \le C(a)\|f\|_2^2.
\ee
\end{Theorem}
Clearly, the assumptions of Theorem \ref{CT5'} can be written in a shorter form: $Y_N$ is a $N$-dimensional subspace of $L_4(\Omega,\mu)$.

Let us now consider the case $L_\infty$, i.e. the case of the uniform norm. Assume that
$\Omega \subset \R^d$ is a compact set and $Y_N:=\sp(u_1(\bx),\dots,u_N(\bx))$ with $\{u_i(\bx)\}_{i=1}^N$ being an orthonormal basis of continuous functions for $Y_N$. Then it is easy to see that for an $\e>0$ we can find $\Omega_M=\{\bx^j\}_{j=1}^M$ such that for all $g$ of the form $g=u_i$, $1\le i\le N$, we have
\be\label{C7}
\left|\|g\|_\infty - \|g\|_{\Omega_M}\right| \le \e.
\ee
We derive from here the following analog of (\ref{C6}): for any $\de>0$ there exists $M=M(\de)$ such that for all $f\in Y_N$
\be\label{C8}
\left|\|f\|_\infty - \|f\|_{\Omega_M}\right| \le \de \|f\|_\infty.
\ee
\begin{Remark}\label{CR1}
We will need a set $\Omega_M$ such that both (\ref{C6}) and (\ref{C8}) are satisfied.
It is easy to see that under assumption of continuity of functions in $Y_N$ on $\Omega=[0,1]^d$ and $\mu$ being the Lebesgue measure on a compact $\Omega$ we can
achieve both (\ref{C6}) and (\ref{C8}) by dividing $[0,1]^d$ into small enough cubes of the same volume.
\end{Remark}
We now prove the following result from \cite{KT169}.
\begin{Theorem}\label{CT6}
 Let  $\Omega := [0,1]^d$. Assume that a real subspace $Y_N\subset \C(\Omega)$ satisfies the Nikol'skii-type inequality: for any $f\in Y_N$
 \be\label{C9}
 \|f\|_\infty \le H(N)\|f\|_2,\quad \|f\|_2 := \left(\int_\Omega |f(\bx)|^2d\mu\right)^{1/2},
 \ee
 where $\mu$ is the Lebesgue measure on $\Omega$.
 Then for any $a>1$ there exists a set $S_m=\{\xi^j\}_{j=1}^m\subset \Omega$ with the property:  $m \le a N$ and
 for any $f\in Y_N$ we have
$$
 \|f\|_\infty \le C(a)H(N)\max_{1\le j\le m} |f(\xi^j)|,
$$
where $C(a)$ is a  positive constant.

 \end{Theorem}
 \begin{proof} The proof consists of two steps. First, using (\ref{C6}) with $\de=1/2$, we find a discrete set $\Omega_M$ such that for any $f\in Y_N$ we have
 \be\label{C10}
| \|f\|_{L_2(\Omega)}^2 - \|f\|_{L_2(\Omega_M)}^2| \le  \|f\|_{L_2(\Omega)}^2/2.
\ee
 Second, we consider a new space $Y_N(\Omega_M)$ which consists of all $f\in Y_N$ restricted to the set $\Omega_M$. Introduce a probability measure $\nu$ on $\Omega_M$ by
 $\nu(\bx^j)=1/M$, $j=1,\dots,M$. Our assumption that $Y_N$ satisfies the Nikol'skii inequality
 (\ref{C9}) and the relation (\ref{C10}) imply that the $Y_N(\Omega_M)$ also satisfies the Nikol'skii inequality. Applying Theorem \ref{CT5} we find a subset $S_m \subset \Omega_M$
 with $m\le aN$ such that
 \be\label{C11}
 \|f\|_{L_\infty(\Omega_M)} \le C'(a)H(N)\|f\|_{L_\infty(S_m)}.
 \ee
 By Remark \ref{CR1} we can claim that $\Omega_M$ guarantees simultaneously (\ref{C6}) and (\ref{C8}). Then by (\ref{C8}) with $\de=1/2$ we obtain from (\ref{C11})
 $$
 \|f\|_\infty \le C''(a)H(N)\|f\|_{L_\infty(S_m)}.
$$
This completes the proof of Theorem \ref{CT6}.
 \end{proof}

\subsection{The Marcinkiewicz theorem and sparse approximation}\label{sec2.4}
We now give some general remarks on the case $q=2$, which illustrate the problem.  We describe the properties of the subspace $X_N$ in terms of a system $\U_N:=\{u_i\}_{i=1}^N$ of functions such that
$X_N = \sp\{u_i, i=1,\dots,N\}$. In the case $X_N \subset L_2$ we assume that
the system is orthonormal on $\Omega$ with respect to measure $\mu$. In the case of real functions   we associate with $x\in\Omega$ the matrix $G(x) := [u_i(x)u_j(x)]_{i,j=1}^N$. Clearly, $G(x)$ is a symmetric positive semi-definite matrix of rank $1$.
It is easy to see that for a set of points $\xi^k\in \Omega$, $k=1,\dots,m$, and $f=\sum_{i=1}^N b_iu_i$ we have
\be\label{Gb}
 \sum_{k=1}^m\la_k f(\xi^k)^2 - \int_\Omega f(x)^2 d\mu = {\mathbf b}^T\left(\sum_{k=1}^m \la_k G(\xi^k)-I\right){\mathbf b},
\ee
where ${\mathbf b} = (b_1,\dots,b_N)^T$ is the column vector and $I$ is the identity matrix. Therefore,
the $\cM^w(m,2)$ problem is closely connected with a problem of approximation (representation) of the identity matrix $I$ by an $m$-term approximant with respect to the system $\{G(x)\}_{x\in\Omega}$. It is easy to understand that under our assumptions on the system $\U_N$ there exist a set of nodes $\{\xi^k\}_{k=1}^m$ and a set of weights $\{\la_k\}_{k=1}^m$, with $m\le N(N+1)/2$ such that
$$
I = \sum_{k=1}^m \la_k G(\xi^k)
$$
and, therefore, we have for any $X_N \subset L_2$ that
\[
X_N \in \cM^w\big(\tfrac{N(N+1)}2,2,0\big).
\]
For the alternative proof see  Theorem \ref{gfT1} in case $q=2$.

As we have seen  the Marcinkiewicz-type
discretization problem in $L_2$ is closely connected with approximation of the identity matrix $I$ by an $m$-term approximant of the form $\frac{1}{m}\sum_{k=1}^m G(\xi^k)$ in the operator norm from $\ell^N_2$ to $\ell^N_2$ (spectral norm).
In a similar way, the Marcinkiewicz-type discretization problem with weights (in $L_2$) is closely connected with approximation of the identity matrix $I$ by an $m$-term approximant of the form $ \sum_{k=1}^m \lambda_k G(\xi^k)$ in the operator norm from $\ell^N_2$ to $\ell^N_2$.
Hence, one can study the following sparse approximation problem.

Let the system $\{u_i(x)\}_{i=1}^N$ be orthonormal and bounded. Then
\be\label{5.2b'}
w(x):=\sum_{i=1}^N u_i(x)^2 \le B.
\ee
Consider
the dictionary
$$
\D^u := \{g_x\}_{x\in\Omega},\quad g_x:= G(x)B^{-1},\quad G(x):=[u_i(x)u_j(x)]_{i,j=1}^N.
$$
Then condition (\ref{5.2b'}) assures  that for the Frobenius norm of $g_x$ one has
\[
\|g_x\|_F = w(x)B^{-1} \le 1.
\]
Our assumption on the orthonormality of the system $\{u_i\}_{i=1}^N$ gives
$$
I = \int_\Omega G(x)d\mu = B\int_\Omega g_x d\mu,
$$
which implies that $I/B \in A_1(\D^u)$, where $A_1(\D^u)$ is the closure
of the convex hull of the dictionary $\D^u$.

We now comment on the use of greedy approximation approach to obtain a deterministic construction
of $\{\xi^\nu,\la_\nu\}_{\nu=1}^m$ providing exact discretization for all $f\in X_N$.
We use the Weak Orthogonal Greedy Algorithm (Weak Orthogonal Matching Pursuit) for $m$-term approximation, which is defined as follows
(see \cite{Tbook2}).

 {\bf Weak Orthogonal Greedy Algorithm (WOGA).} Let $t\in (0,1]$ be a weakness parameter. We define
 $f^{o,t}_0 :=f$. Then for each $m\ge 1$ we inductively define:

(1) $\varphi^{o,t}_m \in \D$ is any element satisfying
$$
|\langle f^{o,t}_{m-1},\varphi^{o,t}_m\rangle | \ge t
\sup_{g\in \D} |\langle f^{o,t}_{m-1},g\rangle |.
$$

(2) Let $H_m^t := \sp (\varphi_1^{o,t},\dots,\varphi^{o,t}_m)$ and let
$P_{H_m^t}(f)$ denote an operator of orthogonal projection onto $H_m^t$.
Define
$$
G_m^{o,t}(f,\D) := P_{H_m^t}(f).
$$

(3) Define the residual after $m$th iteration of the algorithm
$$
f^{o,t}_m := f-G_m^{o,t}(f,\D).
$$

In the case $t=1$ the   WOGA is called the Orthogonal
Greedy Algorithm (OGA).

It is clear from the definition of the WOGA that in case of a finite dimensional Hilbert space $H$ it terminates after $M:=\dim H$ iterations. Consider the Hilbert space
$H^u$ to be a closure in the Frobenius norm of  $\sp\{g_x, x\in\Omega\}$ with the inner product generated by the Frobenius norm: for $A=[a_{i,j}]_{i,j=1}^N$ and
$B=[b_{i,j}]_{i,j=1}^N$
$$
\<A,B\> = \sum_{i,j=1}^N a_{i,j}b_{i,j}
$$
in case of real matrices (with standard modification in case of complex matrices).
We apply the WOGA in
the Hilbert space $H^u$  with respect to the dictionary $\D^u$. The above remark shows that it
provides us a constructive proof of Theorem \ref{gfT1} in case $q=2$.

Under additional assumptions on the system $\{u_i\}$ we can obtain some constructive results for the Marcinkiewicz-type
discretization problem in $L_2$. We use the following greedy algorithm.

{\bf Relaxed Greedy Algorithm (RGA).}  Let $f^r_0:=f$
and $G_0^r(f):= 0$. For a  function $h$ from a real Hilbert space $H$, let $g=g(h)$ denote the function from $\D$, which
maximizes $\langle h,g\rangle$ (we assume the existence of such an element).  Then, for each $m\ge 1$, we inductively define
$$
 G_m^r(f):=
\left(1-\frac{1}{m}\right)G_{m-1}^r(f)+\frac{1}{m}g(f_{m-1}^r), \quad
f_m^r:= f-G_m^r(f).
$$
We make use of the known approximation error of the RGA (see \cite{Tbook2}, p.90).
For a dictionary $\D$ in a Hilbert space $H$ with an inner product $\<\cdot,\cdot\>$,  $A_1(\D)$   denotes the closure
of the convex hull of the dictionary $\D$.
\begin{Theorem}\label{T4.1g} For the Relaxed Greedy Algorithm we have, for each $f\in
A_1(\D)$, the estimate
 $$
\|f-G_m^r(f)\|\le \frac{2}{\sqrt{m}},\quad m\ge 1.
$$
\end{Theorem}

We impose the following restriction on the system $\{u_i\}$: $w(x) \le Nt^2$, i.e.,  $B=Nt^2$.

Using the RGA, we apply
Theorem \ref{T4.1g} for any $m\in \N$ to constructively find points $\xi^1,\dots,\xi^m$ such that
\[
\left\|\frac{1}{m}\sum_{k=1}^m G(\xi^k)-I\right\|_F \le 2Nt^2 m^{-1/2}.
\]
Therefore, using the inequality $\|A\|\le \|A\|_F$ and relation (\ref{Gb}) we arrive at the following result.

\begin{Theorem}[{\cite[Proposition~5.1]{VT159}}] \label{P5.1g} Let $\{u_i\}_{i=1}^N$ be an orthonormal system, satisfying condition {\bf E}. Then there exists a constructive set $\{\xi^j\}_{j=1}^m \subset \Omega$ with $m\le C(t)N^2$
such that for any $f=\sum_{i=1}^N c_iu_i$ we have
$$
\frac{1}{2}\|f\|_2^2 \le \frac{1}{m} \sum_{j=1}^m f(\xi^j)^2 \le \frac{3}{2}\|f\|_2^2.
$$
\end{Theorem}

\section{ Exact weighted discretization}
\label{Ex}

 \subsection{A general result on exact recovery and numerical integration}
 
We begin with a simple useful result which is well-known in many special cases.

\begin{Proposition}\label{gfP1} Suppose $\{u_i(x)\}_{i=1}^N$ is linearly independent  system of functions on $\Omega$. Then there exist a set of points $\{\xi^j\}_{j=1}^N\subset \Omega$ and a set of functions $\{\psi_j(x)\}_{j=1}^N$ such that for any $f\in X_N:=\sp(u_1(x),\dots,u_N(x))$ we have
$$
f(x) = \sum_{j=1}^N f(\xi^j)\psi_j(x).
$$
\end{Proposition}
\begin{proof} For points $x^1,\dots,x^k$ consider the matrix $U(x^1,\dots,x^k) := [u_i(x^j)]_{i,j=1}^k$ with elements $u_i(x^j)$.
\begin{Lemma}\label{gfL1} Under the  conditions of Proposition \ref{gfP1} there exists a set of points $\{\xi^j\}_{j=1}^N\subset \Omega$ such that $D(\xi^1,\dots,\xi^N):= \det U(\xi^1,\dots,\xi^N) \neq 0$.
\end{Lemma}
\begin{proof}
 We prove this lemma by induction.
Indeed, by the linear independence assumption we find $\xi^1$ such that $u_1(\xi^1)\neq 0$. Suppose $2\le k\le N$ and we found a set $\{\xi^j\}_{j=1}^{k-1}$ such that
$D(\xi^1,\dots,\xi^{k-1})\neq 0$. Consider the function $D(\xi^1,\dots,\xi^{k-1},x)$, $x\in \Omega$. This function is a nontrivial linear combination of $u_1(x),\dots,u_k(x)$.
Therefore, there exists $\xi^k\in\Omega$ such that $D(\xi^1,\dots,\xi^{k})\neq 0$.
This completes the proof of the  existence of points $\{\xi^j\}_{j=1}^N\subset \Omega$ such that $D(\xi^1,\dots,\xi^N) \neq 0$.
\end{proof}
Let $f \in X_N$. Then $f$ has a unique
representation $f(x) = \sum_{i=1}^N b_iu_i(x)$. The set of coefficients $\bb:=(b_1,\dots,b_N)$
is uniquely determined from the linear system
$$
(f(\xi^1),\dots,f(\xi^N)) = \bb U(\xi^1,\dots,\xi^N)
$$
and each $b_i$ is a linear combination of $f(\xi^j)$, $j=1,\dots,N$.
This completes the proof of Proposition \ref{gfP1}.
\end{proof}

As a direct corollary of Proposition \ref{gfP1} we obtain the following result on exact
numerical integration.

\begin{Proposition}\label{gfP2} Suppose $\{u_i(x)\}_{i=1}^N$ is a linearly independent  system of integrable functions  with respect to the  measure $\mu$ on $\Omega$. Then there exist a set of points $\{\xi^j\}_{j=1}^N\subset \Omega$ and a set of weights $\{\lambda_j\}_{j=1}^N$ such that for any $f\in X_N:=\sp(u_1(x),\dots,u_N(x))$ we have
$$
\int_\Omega f(x)d\mu = \sum_{j=1}^N \lambda_j f(\xi^j) .
$$
\end{Proposition}

Proposition \ref{gfP2} shows that for any $N$-dimensional subspace of integrable functions we can find an exact cubature formula with $N$ nodes. However, it is not true
for numerical integration by the Quasi-Monte Carlo methods, i.e. by methods with equal weights. Let $\alpha \in (0,1)$ be an irrational number. A trivial example of $\Omega=[0,1]$, $\mu$ is the Lebsgue measure, and $f(x)=1/\alpha$ for $x\in [0,\alpha)$, $f(x) = -(1-\alpha)^{-1}$ for $x\in [\alpha,1]$, shows that there is no Quasi-Monte Carlo quadrature formula, which integrates $f$ exactly.

\subsection{General exact weighted discretization results}

For simplicity, we shall use the notation $L_q(\Omega, \mu)$, or $L_q(\Omega)$ or simply $L_q$ to denote the Lebesgue $L_q$-space defined with respect to the measure $\mu$ on $\Omega$, whenever it does not  cause any confusion from the context.

We begin with a general result establishing an exact weighted discretization theorem in $L_q(\Omega, \mu)$ for  a general measure space $(\Omega,\mu)$ and  even exponent $q$  with at most
$$
M(N,q):= {N+q-1 \choose q}=\frac{(N+q-1)!}{q!(N-1)!}\asymp N^q
$$
nodes, where $N$ is the dimension of the space. We also give an example of a space with discrete $\Omega$ showing that the number ${N+q-1 \choose q}$ cannot be improved.

\begin{Theorem}\label{gfT1}
	Let $q$ be an even positive integer, $N\in\N$, and $M:=M(N,q)$.  For every $N$-dimensional real subspace $X_N\subset L_q(\Omega,\mu)$ we have  that 
 $X_N \in \cM^w\left(M,q,0\right)$.
\end{Theorem}

We point out that under some extra conditions on $\Omega$ a stronger result ensuring positivity of the weights will be given in Corollary~\ref{thm:positive general} of Section~\ref{sec:3-1}.

\begin{proof}   For $\bk=(k_1,\dots,k_N)\in \Z^N_+$ denote $u_\bk:= u_1^{k_1}\cdots u_N^{k_N}$, where $X_N=\sp\{u_1,\dots,u_N\}$. Consider the linear space
$$
X_N(q) :=\sp\{u_\bk:\, \bk\in K(N,q)\},
$$
where $K(N,q):=\{(k_1,\dots,k_N)\in\Z_+^N:k_1+\dots+k_N=q\}$.
Then $\dim(X_N(q)) \le M:= M(N,q)$ and $X_N(q) \subset L_1(\Omega,\mu)$.
Proposition \ref{gfP2} implies that there exist a set of points $\{\xi^j\}_{j=1}^M\subset \Omega$ and a set of weights $\{\lambda_j\}_{j=1}^M$ such that for any $f\in X_N(q)$ we have
$$
\int_\Omega f(x)d\mu = \sum_{j=1}^M \lambda_j f(\xi^j) .
$$
In particular, this implies that for any $f\in X_N$ we have
$$
\int_\Omega f(x)^qd\mu = \sum_{j=1}^M \lambda_j f(\xi^j)^q .
$$
\end{proof}	

Let $\Omega_M=\{\xi^j\}_{j=1}^M$ be a discrete set with the probability measure $\mu(\xi^j)=1/M$, $j=1,\dots,M$.

\begin{Theorem}\label{gfT2} Let $q$ be an even positive integer, $N\in\N$ and $M:=M(N,q)$.
 There exist a discrete set $\Omega_M$ and an $N$-dimensional real subspace $X_N\subset L_q(\Omega_M)$ such that $X_N\notin \cM^w(m,q,0)$ for any $m<M$.
\end{Theorem}
\begin{proof}
	We begin with some preliminaries.
	Suppose $\{u_j\}_{j=1}^N$ is a basis of $X_N$, where $X_N\subset L_q( \Omega_M,\mu)$. We have $X_N \in \cM^w\left(m,q,0\right)$ if and only if there exist $\xi^\nu\in \Omega_M$ and $\lambda_\nu\in\R$, $\nu=1,\dots,m$, such that for any $b_j\in\R$, $j=1,\dots,N$, with $f=\sum_{j=1}^N b_j u_j$ we have
	\begin{gather} 	\nonumber
	0=\int_{ \Omega_M}|f|^qd\mu - \sum_{\nu=1}^m\lambda_\nu|f(\xi^\nu)|^q
	= \int_{ \Omega_M}f^qd\mu - \sum_{\nu=1}^m\lambda_\nu f(\xi^\nu)^q \\
	= \sum_{(k_1,\dots,k_N)\in K(N,q)}\frac{q!}{k_1!\dots k_N!}
	\left(\int_{ \Omega_M} \prod_{j=1}^N u_j^{k_j} d\mu -\sum_{\nu=1}^m \lambda_\nu \prod_{j=1}^N u_j(\xi^\nu)^{k_j}  \right)
	\prod_{j=1}^N b_j^{k_j}, \label{eqn:q diff}
	\end{gather}
	where, as above $K(N,q):=\{(k_1,\dots,k_N)\in\Z_+^N:k_1+\dots+k_N=q\}$. Due to linear independence of multivariate monomials, \eqref{eqn:q diff} holds for any $b_j$ if and only if
	\be\label{eqn:K reduction1}
	\sum_{\nu=1}^m \lambda_\nu \prod_{j=1}^N u_j(\xi^\nu)^{k_j}=\int_ {\Omega_M} \prod_{j=1}^N u_j^{k_j} d\mu,
	\quad (k_1,\dots,k_N)\in K(N,q).
	\ee
Denote  by $\P(N,q)$ the space of homogeneous algebraic polynomials in $N$ variables $x_1,\dots,x_N$ of degree $q$:
$$
\P(N,q):=\text{span} \{ x_1^{k_1}\cdots x_N^{k_N}:\, \bk= (k_1,\dots,k_N)\in K(N,q)\}.
$$
Then $\dim(\P(N,q)) = M$ and by Lemma \ref{gfL1} with $u_\bk(\bx):= x_1^{k_1}\cdots x_N^{k_N}$, $\bx=(x_1,\dots,x_N)$, $\bk= (k_1,\dots,k_N)$, $\bk\in K(N,q)$  there exists a set of points $\{\xi^\nu\}_{\nu=1}^M$ such that $D(\xi^1,\dots,\xi^M):= \det U(\xi^1,\dots,\xi^M) \neq 0$. Define $Y_N$ as a restriction of $\P(N,q)$ onto the
set $\Omega_M:= \{\xi^\nu\}_{\nu=1}^M$. Introduce a probability
measure $\mu$ on $\Omega_M$ by $\mu(\{\xi^\nu\})=1/M$, $\nu=1,\dots,M$.
Then on one hand from the definition of $\mu$ we obtain for any $f\in Y_N$ that
\be\label{gf1'}
\int_{\Omega_M} fd\mu = \frac{1}{M} \sum_{\nu=1}^M f(\xi^\nu).
\ee
On the other hand,  define the space  $X_N:= \sp\Bigl\{u_j\Bl|_{\Omega_M}:\  \  j=1,2,\cdots, N\Bigr\}$ of functions on $\Omega_M$ with measure $\mu$, where $u_j(\mathbf{x})=x_j$, $j=1,\dots, N$.
 Assume that for any $f\in X_N$,
\be\label{gf2'}
\int_{\Omega_M} f^qd\mu = \sum_{\nu=1}^M \la_\nu f(\xi^\nu)^q.
\ee
Then by (\ref{eqn:K reduction1}) and the choice of $\{\xi^\nu\}_{\nu=1}^M$ the set of weights satisfying (\ref{gf2'}) is unique. Relation (\ref{gf1'}) shows that $\la_\nu =1/M$,
$\nu=1,\dots,M$. Therefore, none of these $\la_\nu$ is equal to zero.
This argument completes the proof.
\end{proof}

Now we show that one cannot obtain an exact weighted Marcinkiewicz-type theorem in $L_q$ when $q$ is not an even integer.
\begin{Proposition}
	Consider $X_2:=\{\alpha \sin t+\beta \cos t: \alpha,\beta\in\R\}$ as a subspace of $L_q(\T)$, where $\T$ is the unit circle and $q$ is not an even integer, $1\le q<\infty$. Then $X_2\notin \cM^w(m,q,0)$ for any $m\in\N$.
\end{Proposition}
\begin{proof}
	Suppose to the contrary that for some distinct $\xi^\nu\in[0,2\pi)$ and non-zero $\lambda_\nu$, $\nu=1,\dots,m$, we have
	\be\label{eqn:circle example1}
	\frac1{2\pi}\int_0^{2\pi} |f(t)|^qdt=\sum_{\nu=1}^m \lambda_\nu |f(\xi^\nu)|^q
	\ee
	for every $f\in X_2$. For any $\theta\in\R$ we define $f_\theta(t):=\sin(\theta-t)\in X_2$ and note that
	\[
	\frac1{2\pi}\int_0^{2\pi} |f_\theta(t)|^qdt=\frac1{2\pi}\int_0^{2\pi} |\sin t|^qdt=c(q),
	\]
	where $c(q)>0$ depends only on $q$ and does not depend on $\theta$. Therefore, by~\eqref{eqn:circle example1} we have
	\[
	c(q)=\sum_{\nu=1}^m \lambda_\nu |f_\theta(\xi^\nu)|^q=\sum_{\nu=1}^m \lambda_\nu |\sin(\theta-\xi^\nu)|^q=:G(\theta),
	\]
	i.e., $G$ is a constant function.  However, if $n-1<q\le n$, $n\in\N$ and $q$ is not an even integer, then the function $g(t):=|\sin t|^q$ is infinitely smooth when $t\in(0,\pi)$ while $g^{(n)}(0)$ does not exist. As $g$ is a $\pi$-periodic function and
	$G(\theta)=\sum_{\nu=1}^m\lambda_\nu g(\theta-\xi^\nu)$, without loss of generality we may assume that $|\xi^\nu-\xi^{v'}|\neq \pi$ for $1\leq v, v'\leq m$.
	It then follows  that $G^{(n)}(\xi^1)$ does not exist, which contradicts  the fact that $G$ is constant.
\end{proof}

\subsection{Relation between exact weighted discretization  and  recovery problem}
We now discuss a connection between exact weighted discretization theorem in $L_2$ and exact recovery.

\begin{Proposition}\label{gfP3} Let $N$-dimesional real $X_N$ be a subspace of $L_2(\Omega)$. Suppose sets of points $\{\xi^j\}_{j=1}^m$ and of weights $\{\lambda_j\}_{j=1}^m$ are such that for any $f\in X_N$ we have
\be\label{gf1}
\|f\|_2^2 = \sum_{j=1}^m \lambda_j f(\xi^j)^2.
\ee
Then for any orthonormal basis $\{u_i\}_{i=1}^N$ of $X_N$ we have for any $f\in X_N$
\be\label{gf2}
f(x) = \sum_{j=1}^m \lambda_j f(\xi^j)D(x,\xi^j),\quad D(x,y):= \sum_{i=1}^N u_i(x)u_i(y).
\ee
\end{Proposition}
\begin{proof} For $f,g\in X_N$ denote
$$
\<f,g\> := \int_\Omega f(x)g(x)d\mu.
$$
Using identity $4\<f,g\>= \|f+g\|_2^2 - \|f-g\|_2^2$, we obtain from (\ref{gf1}) that for any $f,g\in X_N$
$$
\<f,g\> = \sum_{j=1}^m \lambda_j f(\xi^j)g(\xi^j).
$$
This implies
$$
f(x) = \<f,D(x,\cdot)\> = \sum_{j=1}^m \lambda_j f(\xi^j)D(x,\xi^j),
$$
which completes the proof.
\end{proof}

Note that if for some orthonormal bases $\{u_i\}_{i=1}^N$ of $X_N$ there exist
sets of points $\{\xi^j\}_{j=1}^m$ and of weights $\{\lambda_j\}_{j=1}^m$ such that
(\ref{gf2}) holds for all $f\in X_N$ then also (\ref{gf1}) holds for all $f\in X_N$.
Indeed,
$$
\sum_{j=1}^m \lambda_j f(\xi^j)^2 = \sum_{j=1}^m \lambda_j f(\xi^j)\int_\Omega f(y)D(\xi^j,y)d\mu
$$
$$
= \int_\Omega f(y)\left(\sum_{j=1}^m \lambda_j f(\xi^j) D(\xi^j,y)\right)d\mu = \int_\Omega f(y)^2d\mu = \|f\|_2^2.
$$

Given    a finite subset $Q$ of $\Z^d$,  we denote
$$
\Tr(Q):= \Bigl\{f: f=\sum_{\bk\in Q}c_\bk e^{i(\bk,\bx)},\   \  c_{\mathbf{k}}\in \mathbb{C},\  \  \bk\in Q\Bigr\}.
$$
Also, given $\bN=(N_1,\dots,N_d)\in\mathbb{Z}_+^d$, we write  $\Tr(\bN)$ for the set  $\Tr(\Pi(\bN))$ with $$\Pi(\bN):=\bl\{ \bk=(k_1,  \cdots, k_d)\in\Z_+^d:\  \   |k_j|\leq N_j,\   \ j=1,\cdots, d\br\}. $$

\begin{Proposition}\label{gfP4} Let $\bN=(N_1,\dots,N_d)\in\mathbb{Z}_+^d$. Suppose a cubature formula
$\Lambda_m(\cdot,\xi)$ is exact for $\Tr(2\bN)$.
 Then $m\ge \prod_{j=1}^d (2N_j+1)$.
\end{Proposition}
\begin{proof} The proof is by contradiction. Suppose $m< \prod_{j=1}^d (2N_j+1)$.
Then, using the fact $\dim \Tr(\bN) = \prod_{j=1}^d (2N_j+1)$, we find a non-zero $f\in \Tr(\bN)$ such that $f(\xi^\nu)=0$, $\nu=1,\dots,m$. Then $|f|^2\in \Tr(2\bN)$ and
$\int_{\T^d}|f(\bx)|^2d\mu \neq 0$ but $\Lambda_m(|f|^2,\xi)=0$. We got a contradiction, which proves Proposition~\ref{gfP4}.
\end{proof}

\subsection{Exact weighted discretization for spaces of spherical harmonics}

Theorems \ref{gfT1} and \ref{gfT2} solve the problem of optimal behavior of $m$ for
exact weighted discretization in the general setting. Theorem \ref{gfT1} shows that in case of even natural number $q$ we always have $X_N\in \cM^w(M(N,q),q,0)$.
Theorem \ref{gfT2} shows that the parameter $m=M(N,q)$ is the best possible one in a general setting. However, it is well known that for specific subspaces $X_N$ the growth of $m$ allowing exact weighted discretization may be much slower than
$M(N,q)\asymp N^q$. In this subsection we show that the subspaces of spherical harmonics are as bad (in the sense of order) as the worst subspaces.

Let $ \mathcal{H}_n^2$ denote the space of spherical harmonics of degree $n$ on the unit sphere $\mathbb{S}^2$ of $\mathbb{R}^3$. It is known that $\text{dim}\  (\mathcal{H}_n^2)=2n+1$.  Let $\{ Y_{n,j}\}_{j=1}^{2n+1}$ denote   an orthonormal basis  in $ \mathcal{H}_n^2$.  Denote by  $d\sigma$  the surface Lebesgue measure on $\mathbb{S}^2$ normalized by $\int_{\mathbb{S}^2} d\sigma=1$.

\begin{Theorem}\label{thm-1-1}
	Assume that there exist   distinct points $\xi_1, \cdots, \xi_m\in\mathbb{S}^2$ and real numbers $\lambda_1,\cdots, \lambda_m$  such that
	\begin{equation}\label{1}\int_{\mathbb{S}^2} |f(x)|^2 \, d\sigma(x) =\sum_{j=1}^m  \lambda_j |f(\xi_j)|^2,\   \   \ \forall f\in \mathcal{H}_n^2.\end{equation}
	Then
	$m\ge \frac {n(n+1)}2$. \end{Theorem}

For the proof of Theorem \ref{thm-1-1}, we need the following identity on ultraspherical polynomials, which can be found in \cite[p. 39, (5.7)]{A1}.
\begin{Lemma} For each positive integer $n$ and every $\lambda>0$,
	\begin{equation}\label{2} \Bl|C_n^\lambda (t)\Br|^2 =\sum_{j=0}^n b^{\lambda}_{n,j}\f {2j+\lambda}{\lambda} C_{2j}^{\lambda}(t),\end{equation}
	where
	$$ b_j^{\lambda} =\frac{\lambda}{n+\lambda+j}\frac{ (\lambda)_{n-j} ( (\lambda)_j)^2 (2\lambda)_{n+j} (2j)!} {(n-j)! (j!)^2 (\lambda)_{n+j} (2\lambda)_{2j} },\   \  0\leq j\leq n,$$
	and
	$(a)_j=a(a+1) \cdots (a+j-1)$.
\end{Lemma}

\begin{proof}[Proof of Theorem \ref{thm-1-1}]
	First, we show that
	\begin{equation}\label{1-3-0}\sum_{j=1}^m \lambda_j =1.\end{equation}
	Recall that the function $(2j+1) C_j^{1/2}(x\cdot y)$, $x, y\in\mathbb{S}^2$ is the reproducing kernel of the space $ \mathcal{H}_n^2$.
	Thus, for each $x\in\mathbb{S}^2$, we have
	\begin{align*}
	\int_{\mathbb{S}^2} |C_n^{1/2} (x\cdot y)|^2 \, d\sigma(y) =\sum_{j=1}^m  \lambda_{j} |C_n^{1/2}(x\cdot \xi_j)|^2 .
	\end{align*}	
	Integrating over $x\in\mathbb{S}^2$ then gives
	\begin{align*}
	\int_{\mathbb{S}^2}	\int_{\mathbb{S}^2} |C_n^{1/2} (x\cdot y)|^2 \, d\sigma(y)d\sigma(x) & =\sum_{j=1}^m \lambda_{j} \int_{\mathbb{S}^2}|C_n^{1/2}(x\cdot \xi_j)|^2\, d\sigma(x) \\
	&=\bl(\sum_{j=1}^m  \lambda_j\bigr)\int_{\mathbb{S}^2}\int_{\mathbb{S}^2} |C_n^{1/2} (x\cdot y)|^2 \, d\sigma(x)d\sigma(y).
	\end{align*}
	This implies
	\eqref{1-3-0}.
	
	Next, we show that
	\begin{equation}\label{1-4-0}\int_{\mathbb{S}^2} f(x)\, d\sigma(x) =\sum_{j=1}^m \lambda_{j} f(\xi_j),\   \ \ \forall f\in \bigoplus_{j=0}^{n} \mathcal{H}^2_{2j}.\end{equation}	
	Indeed,  using \eqref{1}, we have 	
	\begin{align*}
	\int_{\mathbb{S}^2} \int_{\mathbb{S}^2} |C^{1/2}_n(x\cdot y)|^2\, d\sigma(x) d\sigma(y)& =\sum_{j=1}^m  \lambda_{j} \int_{\mathbb{S}^2}| C^{1/2}_n(x\cdot \xi_j)|^2 \, d\sigma(x) \\
	&=\sum_{j=1}^m \sum_{k=1}^m \lambda_{j}\lambda_{k} |C^{1/2}_n(\xi_j\cdot \xi_k)|^2,  \end{align*}
	which, using \eqref{2} and \eqref{1-3-0}, equals
	\begin{align*}	& =b_{n,0}^{\frac12}+ \sum_{i=1}^n b_{n,i}^{\frac12}\sum_{j=1}^m \sum_{k=1}^m \lambda_{j}\lambda_{k}   (4i+1) C_{2i}^{1/2} (\xi_j\cdot \xi_k).\end{align*}
	It then follows by the addition formula for spherical harmonics that
	\begin{align*}\int_{\mathbb{S}^2} \int_{\mathbb{S}^2} |C^{1/2}_n(x\cdot y)|^2\, d\sigma(x) d\sigma(y) 	&=b_{n,0}^{\frac12}+ \sum_{i=1}^n b_{n,i}^{\frac12}\sum_{\ell=1}^{4i+1} \sum_{j=1}^m \sum_{k=1}^m  \lambda_{j}\lambda_{k}   Y_{2i,\ell}(\xi_j) Y_{2i,\ell}(\xi_k)\\
	&=b_{n,0}^{\frac12}+ \sum_{k=1}^n b_{n,k}^{\frac12}\sum_{j=1}^{4k+1}\Bl| \sum_{i=1} ^m \lambda_{i}   Y_{2k,j}(\xi_i) \Br|^2.
	\end{align*}
	Note  that by \eqref{2},
	$$ b_{n,0}^{\frac12} =	\int_{\mathbb{S}^2} \int_{\mathbb{S}^2} |C^{1/2}_n(x\cdot y)|^2\, d\sigma(x) d\sigma(y).$$
	It follows that  for $1\leq k\leq n$ and $1\leq j\leq 4k+1$,
	$$\sum_{i=1}^m \lambda_{i} Y_{2k,j}(\xi_i)=0.$$
	This together with \eqref{1-3-0} implies \eqref{1-4-0}.

	Finally, we show that
	$
	m \ge (n_0+1) (2n_0+1),
	$
	where $n_0$ is the integer part of $n/2$.
	To see this, note that
	for each $f\in V_n: =\bigoplus_{0\leq j\leq n/2}\mathcal{ H}_{2j}^2$, we have
	$ |f|^2 \in \bigoplus_{0\leq j\leq n}\mathcal{ H}_{2j}^2.$
	Thus, using   \eqref{1-4-0}, we obtain
	$$ \int_{\mathbb{S}^2} |f(x)|^2 \, d\sigma(x) =\sum_{i=1}^m \lambda_{i} |f(\xi_i)|^2,\   \  \forall f\in V_n.$$
	In particular, this implies that
	$$m \ge \text{dim} ( V_n)=\sum_{0\leq j\leq n/2} (4j+1)=(n_0+1) (2n_0+1).$$
\end{proof}

\subsection{Exact weighted discretization for  trigonometric polynomials}\label{Ex--}

In this subsection we show that some subspaces of trigonometric polynomials are as bad (in the sense of order) as the worst subspaces.   Recall that given    a finite subset $Q$ of $\Z^d$,
$$
\Tr(Q):= \Bigl\{f: f=\sum_{\bk\in Q}c_\bk e^{i(\bk,\bx)},\   \  c_{\mathbf{k}}\in \mathbb{C},\  \  \bk\in Q\Bigr\}.
$$
We begin with a univariate trigonometric polynomials.
\begin{Theorem} Let $N$ be a given positive integer and let
$$
Q:=\Bl\{j^2:\  \  j=1,2,\cdots, N\Br\}\cup\Bl\{0, 1, 2, \cdots, 2N\Br\}.
$$
Assume that there  are points  $ x_1, \cdots, x_m\in [0,2\pi)$ and real  numbers $\lambda_1,\cdots, \lambda_m$ such that
	\begin{equation}\label{5-2}\f 1{2\pi} \int_0^{2\pi} |f( x)|^2\, dx = \sum_{j=1}^m  \lambda_j |f(x_j)|^2,\   \ \forall f\in \mathcal{T}(Q).\end{equation}
Then
$$
 m\ge N^2\ge \f {(|Q|-1)^2}9.
 $$
\end{Theorem}
\begin{proof}
	Note first that since $\text{Re}(a\overline{b})=\frac1{4}\big(|a+b|^2-|a-b|^2\big)$,  \eqref{5-2} implies that
\be\label{5-2'}
	\f 1{2\pi}\int_0^{2\pi} f(x) \overline{g(x)}\, dx =\sum_{j=1}^m \lambda_j f(x_j) \overline{g(x_j)},\    \  \forall f,g\in \mathcal{T}(Q).
\ee
	Applying this last formula to $f(x)=e^{ijx}$ and $g(x)=e^{ikx}$ with $j,k\in Q$, and using linearity of the integral, we then conclude that
	\begin{equation}\label{5-3}
	\f 1{2\pi}\int_0^{2\pi} h(x)\, dx =\sum_{j=1}^m \lambda_j h(x_j),\   \   \forall h\in \mathcal{T}(Q-Q).
	\end{equation}
	
	Next, we note that  $3N+1-\sqrt{2N}\leq |Q| \leq 3N+1$, and
\begin{eqnarray*}	 Q-Q &\supset&
\Big(\bigcup_{k=1}^N \{ k^2, k^2-1, \cdots, k^2-2N\}\Big)
\\
 &&\cup\Big(\{1^2, 2^2, \cdots, N^2\}\Big).
\end{eqnarray*}
	Since
	$ k^2-(k-1)^2=2k-1<2N$  for $1\leq k\leq N$,
	this together with symmetry  implies that
	\begin{equation}\label{5-4}
	 \Bl\{ \pm j:\   \ j=0,1,2,\cdots, N^2\Br\}\subset Q-Q.
	\end{equation}

Finally, \eqref{5-4} combined with \eqref{5-3} implies that the cubature formula in  \eqref{5-3} is exact for every $f\in\mathcal{T}(N^2)$.
Thus,
by Proposition \ref{gfP4} we obtain the required lower bound.	
\end{proof}
It is worth mentioning that the construction of the set $Q$ satisfying the conditions
$|Q|  \asymp N$ and \eqref{5-4}
 is closely related to the so-called Sidon's sets, see, e.g., \cite{sidon1, sidon2}.

We now give one more example of a subspace of multivariate trigonometric polynomials, which is ``difficult'' for exact weighted discretization. Let
$\cR\Tr(n)$ denote the set of real trigonometric polynomials of degree at most $n$.
For $q=2^s$, $s\in\N$, consider the following subspace
$$
X_N:=\{f(x_1,\dots,x_q)=f_1(x_1)+\cdots+f_q(x_q):\, f_j\in \cR\Tr(n),\, j=1,\dots,q\}.
$$
Then $N=\dim(X_N) = (2n+1)q$. Assume that sets $\{\xi^\nu\}$ and $\{\la_\nu\}$, $\nu=1,\dots,m$ are such that for any $f\in X_N$ we have
$$
(2\pi)^{-q}\int_{\T^q}f(\bx)^qd\bx = \sum_{\nu=1}^m \la_\nu f(\xi^\nu)^q.
$$
Using the form $q=2^s$ and applying $s$ times the argument, which we used above to derive (\ref{5-2'}) from (\ref{5-2}), we obtain that for any $f_j\in \cR\Tr(n)$, $j=1,\dots,q$ we have
$$
(2\pi)^{-q}\int_{\T^q}f_1(x_1)\cdots f_q(x_q)d\bx = \sum_{\nu=1}^m \la_\nu f_1(\xi^\nu_1)\cdots f_q(\xi^\nu_q).
$$
In particular, this implies that the cubature formula with nodes $\{\xi^\nu\}$ and weights $\{\la_\nu\}$, $\nu =1,\dots,m$, is exact for $\Tr(\bN)$, $\bN=(n,\dots,n)$. Therefore, by
Proposition \ref{gfP4} we get $m\ge n^q$.

\section{Exact weighted discretization with constraints on the weights}
\label{constr}

In Section \ref{Ex} we discussed the problem of exact weighted discretization and
related problems of recovery and numerical integration. In that setting we did not impose any restrictions on the weights $\{\la_\nu\}$. In this section we consider numerical integration and exact weighted discretization with additional constraints on
weights $\{\la_\nu\}$. We only consider two natural types of constraint.

{\textsc{Positivity.}} We assume that $\la_\nu \ge 0$, $\nu=1,\dots,m$.

{\textsc{Stability.}} We assume that $\sum_{\nu=1}^m |\la_\nu| \le B$.
\\In Subsection \ref{section stable exact} we  will also consider a more general stability property than the one above.

\subsection{ Exact weighted discretization with positive weights}
\label{sec:3-1}

In this section, we shall  prove that given an  $N$-dimensional subspace of continuous and integrable functions on a sequentially compact space, one can always find an exact positive cubature formula with at most $N$ nodes.

\begin{Theorem}\label{thm-4-1}
	Let $\Omega$ be a sequentially  compact topological  space with the     probability Borel  measure $\mu$. Then for each given    $N$-dimensional  real  linear subspace $X_N$  of $L_1(\Omega,\mu)\cap C(\Omega)$,   there exist a set of $N$ points $\{\xi^1, \cdots, \xi^N\}\subset \Omega$ and  a set of  nonnegative    real numbers $\lambda_1,\cdots, \lambda_N$   such that
	\begin{equation}\label{4-1-0}
	\int_{\Omega} f(x)\, d\mu(x) =\sum_{j=1}^{N}\lambda_j f(\xi^j),\quad  \ \forall f\in X_N.
	\end{equation}
\end{Theorem}

	This theorem guarantees existence of exact positive cubature formula with at most $N$ nodes. One can observe that we actually have $2N$ parameters as both the nodes and the weights can be chosen, while the dimension of the subspace is $N$. Therefore, in many concrete situations reduction of the number of nodes is possible. Perhaps the simplest example is the classical Gaussian quadrature of highest algebraic degree of exactness, see, e.g.~\cite{Krylov}. 

In the  case when $\Omega$ is a compact subset of $\R^d$, and  $X_N$ is  the space of all  real algebraic polynomials in $d$ variables  of total degree at most $n$, Theorem~ \ref{thm-4-1} is known as the  Tchakaloff theorem, and its  proof can be found in \cite{P} (see also \cite{Wi}).  It is worthwhile to point out that
Theorem~ \ref{thm-4-1} here is applicable in  a more general setting, and our  proof is different from that of  the  Tchakaloff theorem in  \cite{P}.

 Theorem \ref{thm-4-1} has two useful corollaries,  the first of which  provides   an exact positive cubature formula with  one more  node (i.e., $N+1$ nodes instead of $N$ nodes)   and the additional property that  the sum of all the weights $\lambda_\nu$ is  $1$.

\begin{Corollary}\label{cor-3-1-FD} Under the conditions of Theorem~ \ref{thm-4-1}, there
	 exist a set of $N+1$ points $\{\xi^1, \cdots, \xi^{N+1}\}\subset \Omega$ and  a set of  nonnegative    real numbers $\lambda_1,\cdots, \lambda_{N+1}$   such that $\sum_{j=1}^{N+1} \lambda_j=1$ and
	\begin{equation}\label{4-1-0}
	\int_{\Omega} f(x)\, d\mu(x) =\sum_{j=1}^{N+1}\lambda_j f(\xi^j),\quad  \ \forall f\in X_N.
	\end{equation}
	
	\end{Corollary}

The proof of Corollary \ref{cor-3-1-FD} is almost identical to that of Theorem \ref{thm-4-1}. The only difference is that one uses  the  Carath\'eodory theorem instead of Lemma~\ref{lem:carath positive} below.

The  second corollary  guarantees  the existence of   an exact weighted discretization theorem
with positive weights and  at most $M(N,q)$ nodes  for each even positive integer $q$ and  each given  $N$-dimensional real linear subspace of $L_q$, where
$$
M(N,q):= {N+q-1 \choose q}=\frac{(N+q-1)!}{q!(N-1)!}\asymp N^q.
$$

Following the proof of Theorem \ref{gfT1}, one can easily deduce from  Theorem~\ref{thm-4-1} the following corollary,  which in particular improves   Theorem \ref{gfT1} in the sense that all the weights $\lambda_\nu$  are nonnegative.

\begin{Corollary}\label{thm:positive general}
	Assume that the conditions of Theorem~ \ref{thm-4-1} are satisfied, and
	   $X_N\subset L_q(\Omega,\mu)$ for some positive integer $q$.  Let  $M:={N+q-1 \choose q}$. Then there exist $\xi^\nu\in\Omega$ and $\lambda_\nu\ge0$, $\nu=1,\dots,M$, such that
	\[
	\int_{\Omega}f^qd\mu = \sum_{\nu=1}^M\lambda_\nu f(\xi^\nu)^q,\quad   \ \forall f\in X_N.
	\]
	In particular, if $q$ is even, then $X_N \in \cM^w_+(M,q,0)$.
\end{Corollary}

   Note that  according to Theorem \ref{gfT2},   the lower bound  ${N+q-1 \choose q}$  for the number of  nodes  in the exact weighted  discretization theorem
  is sharp even without  the positivity assumption.

Now we turn to the proof of Theorem \ref{thm-4-1}. We will use the notation $\conv(E)$ to denote the convex hull of a set $E\subset\R^M$, while $\overline{\text{conv} (E)}$ will denote the closure of $\conv(E)$. We need  the following lemma:
\begin{Lemma}\label{lem:carath positive}
	Suppose that $E\subset\R^M$ and $z\in\conv(E)$. Then one can find $ y_\nu\in E$ and $\lambda_\nu\ge0$, $\nu=1,\dots,M$, satisfying $z=\sum_{\nu=1}^M\lambda_\nu y_\nu$.
\end{Lemma}
\begin{proof}
	This is a corollary from a generalization of the Carath\'eodory theorem. For instance, one can use~\cite[Corollary~17.1.2, p.~156]{Ro} for one-element sets $C_y:=\{y\}$, $y\in E=:I$.
\end{proof}

\begin{proof}[Proof of Theorem \ref{thm-4-1}] Let $u_1,\cdots, u_N$ be a basis of the linear space $X_N$.
		Define the mapping
	$\Phi:\  \  \Omega\to \R^N$ by $\Phi(x)=(u_1(x),\cdots, u_N(x))$, $x\in\Omega$.
By linearity, it is easily seen that relation \eqref{4-1-0} is equivalent to the system of equations:
\begin{equation}\label{3-3-FD}
 \sum_{j=1}^N \lambda_j \Phi(x_j)=\mathbf{a},\   \   x_1,\dots, x_N\in\Omega,\  \  \lambda_1,\dots, \lambda_N\ge 0,
\end{equation}
where $$\mathbf{a}=\Bigl( \int_{\Omega} u_1(x) \, d\mu, \int_{\Omega} u_2(x)\, d\mu, \dots, \int_{\Omega} u_N(x)\, d\mu\Bigr)=\int_{\Omega} \Phi(x) \, d\mu(x).$$
For the proof of \eqref{3-3-FD},  by Lemma \ref{lem:carath positive}, it suffices to prove that $\mathbf{a} \in\text{conv} (E)$, where $E=\Phi(\Omega)\subset \mathbb{R}^N$. 	
	
To this end, we first prove that $\mathbf{a} \in \overline	{\text{conv} (E)}$.   Assume to the contrary that this is not true. Then by the convex  separation theorem in $\mathbb{R}^N$,  we can find $\alpha\in\mathbb{R}^N$ and $t\in\mathbb{R}$ such that  $\alpha \cdot \mathbf{a} >t\ge \sup_{x\in \Omega} \alpha \cdot \Phi(x)$. This gives a contradiction since $\alpha\cdot\mathbf{a}=\int_\Omega \alpha\cdot \Phi(x) d\mu(x)$.
	
	Next, we show that $\mathbf{a} \in 	{\text{conv} (E)}$.  Since $\mathbf{a} \in \overline	{\text{conv} (E)}$, it follows by the Carath\'eodory  theorem that for any positive integer $n$, there exist $\lambda_{n,v}\ge 0$ and $\xi_{n,v}\in\Omega$, $\nu=1,\cdots, N+1$ such that
	$\sum_{\nu=1}^{N+1} \lambda_{n,v}=1$ and
	$$ \Bigl\| \mathbf{a} - \sum_{\nu=1}^{N+1} \lambda_{n,v} \Phi(\xi_{n,v})\Bigr\|\leq n^{-1}.$$
	Since $\Omega$ is sequentially compact, without loss of generality, we may assume that
	$\lim_{n\to\infty}\lambda_{n,v}=\lambda_\nu\ge 0$ and $\lim_{n\to\infty} \xi_{n,v} =\xi_v \in \Omega$, $\nu=1, \dots, N+1$. Then $\sum_{\nu=1}^{N+1}\lambda_{v} =1$ and
	by continuity of the mapping $\Phi: \Omega \to \mathbb{R}^N$,
	$\mathbf{a} =\sum_{\nu=1}^{N+1} \lambda_\nu \Phi(\xi_v)$. This proves that $\mathbf{a} \in \text{conv}(E)$. 	
\end{proof}

 \subsection{Stable exact weighted discretization}\label{section stable exact}

In this subsection we prove that the one-sided Marcinkiewicz-type estimate implies the existence of exact cubature formula with the same number of nodes.

 Let  $\Omega$ be a subset of $\R^d$ equipped with a probability Borel measure $\mu$.
Let  $1\leq p\leq \infty$ and  $X_N\subset L_p(\Omega, \mu)\cap C(\Omega)$ be an $N$-dimensional real subspace.
\begin{Theorem} Assume that there exist a finite subset $W\subset \Omega$  and a set $\{\mu_{\omega}:\  \ \omega\in W\}$ of positive  numbers  such that
	\begin{equation}\label{3-1-0}
	\|f\|_{L_p(\Omega, \mu)} \leq C_1 \Bl(\sum_{\omega\in W} \mu_{\omega} |f(\omega)|^p\Br)^{1/p},\quad   \  \forall f\in X_N,
	\end{equation}
	if $p<\infty$, and
	\[
	\|f\|_{L_\infty(\Omega, \mu)} \leq C_1\sup_{\omega\in W} |f(\omega)|,\quad   \  \forall f\in X_N,
	\]
	if $p=\infty$. Then there exists a sequence of real numbers $\{ \lambda_\omega:\  \ \omega\in W\}$ such that
	$$ \int_{\Omega} f(x)\, d\mu(x) =\sum_{\omega\in W} \lambda_{\omega} f(\omega),\   \   \ \forall f\in X_N,$$
	and
	\begin{equation}\label{3-1}
	\left(\sum_{\omega\in W}   \left|\frac{\lambda_\omega}{\mu_{\omega}}\right|^{p'} \mu_{\omega}\right)^{1/p'}\leq C_1,
	\end{equation}
	where $\tfrac1p+\tfrac1{p'}=1$ if $p\ne1$, while in the case of $p=1$ we obtain
	$$|\lambda_\omega|\leq C_1\mu_\omega \quad \mbox{for all}\quad \omega\in W$$
	instead of~\eqref{3-1}.
\end{Theorem}
\begin{proof}
	We give the proof for the case $p\ne1$ only, the required modifications for $p=1$ are obvious. Denote by $X_N^\ast$ the dual space of $X_N$.  Define  $$ E:=\Bl\{ \sum_{\omega\in W} \lambda_\omega \delta_\omega:\  \   \lambda_\omega\in\R,\   \   \Bigl(\sum_{\omega\in W}    |\lambda_\omega|^{p'} \mu_{\omega}^{1-p'}\Bigr)^{1/p'}\leq C_1\Br\},$$
	where  $\delta_x$ denotes  the linear functional in $ X_N^\ast$  given by $\delta_x(g)=g(x)$,  $g\in X_N$, $x\in\Omega$.  Clearly, it is sufficient to show that the linear functional  $\ell\in X_N^\ast$ given by
	$$\ell(g):=\int_{\Omega}  g(x)\, d\mu(x),\   \  g\in X_N,$$
	lies in the set $E$.   Assume to the contrary that $\ell\notin E$.  It then  follows by the convex separation theorem that there exists a nonzero function  $f\in X_N$ such that
	$$ \Bigl\langle \sum_{\omega\in W} \lambda_\omega \delta_{\omega}, f\Bigr\rangle =\sum_{\omega\in W} \lambda_{\omega} f(\omega) <1< \int_{\Omega} f(x)\, d\mu(x)\leq \|f\|_p$$
	for every sequence $\{\lambda_\omega\}_{\omega\in\Lambda}$ of real numbers  satisfying \eqref{3-1}. Taking supremum over all real  sequences  $\{\lambda_\omega\}_{\omega\in W}$ satisfying \eqref{3-1}, we  obtain
	$$C_1\Bl(\sum_{\omega\in W} \mu_{\omega} |f(\omega)|^p\Br)^{1/p'} < \|f\|_{L_p(\Omega, d\mu)},$$
	which contradicts the condition  \eqref{3-1-0}.
\end{proof}

\section{Marcinkiewicz-type inequality for the hyperbolic cross polynomials for $q=\infty$ }
\label{hyper}

Recall that the  set of hyperbolic polynomials is defined as
$$
\mathcal{T}(N):= \Tr(N,d) := \Bl\{ f:\  \  f=\sum_{\mathbf{k} \in \Gamma(N)} c_{\mathbf{k}} e^{i (\mathbf{k},\bx)}\Br\},$$
where $ \Gamma(N)$ is the hyperbolic cross
$$  \Gamma(N):=\Gamma(N,d) :=\Bl\{\mathbf{k}\in\Z^d:\  \  \   \prod_{j=1}^d \max\{ |k_j|, 1\} \leq N\Br\}.$$
Throughout this section, we define
$$\al_d:=\sum_{j=1}^d \f 1j\qquad\mbox{ and}\qquad
\b_d:=d-\al_d.$$ We use the following notation here. For $\bx\in\T^d$ and $j\in\{1,\dots,d\}$ we denote $\bx^j:=(x_1,\dots,x_{j-1},x_{j+1},\dots,x_d)$.
Our main result in this section can be stated as follows.
\begin{Theorem}\label{thm-2-1} For each $d\in\N$ and each $N\in\N$ there exists a set $W(N,d)$ of at most $C_d N^{\alpha_d}(\log N)^{\beta_d}$ points in $[0, 2\pi)^d$ such that for all $f\in\mathcal{T}(N)$,
	$$ \|f\|_\infty \le C(d) \max_{\bw\in W(N,d) }|f(\bw)|.$$
\end{Theorem}

Theorem \ref{thm-2-1} for $d=1$ is well known (see, for instance, Subsection 6.1 for a detailed discussion).  We prove Theorem \ref{thm-2-1} by induction on $d$. For readers' convenience, first we demonstrate in Subsection \ref{subsection2:1} the step from $d=1$ to $d=2$. Second, we demonstrate in Subsection \ref{subsection2:2} the general step from $d-1$ to $d$. An important ingredient in the proof is the following Bernstein's inequality (see, for instance, \cite{Tmon}):
\begin{Lemma}\label{lem-2-1}
	For each $f\in \mathcal{T}(N)$,
	$$\|f^{(1,\cdots, 1)}\|_\infty \leq C(d) N (\log N)^{d-1} \|f\|_\infty.$$
\end{Lemma}

\subsection{Step from $d=1$ to $d=2$}\label{subsection2:1}
This subsection is devoted to the proof of Theorem \ref{thm-2-1}  for $d=2$. As we already pointed out above Theorem \ref{thm-2-1} is known in the case $d=1$.
For $M\in\N$ define
$$
V_M:=\Bl\{ \f {2\pi j }{M}:\   \   j=0,1,\cdots, M-1\Br\} .
$$
For natural numbers $M$ and $N$ we set
$$
V(M,N,2,j) := \{\bx\in \T^2: x_j \in V_M,\, \bx^j \in W(N,1)\}, \quad j=1,2,
$$
where $2$ stands for dimension.
Finally, define
$$
W:=W_{M,N}:= V(M,N,2,1)\cup V(M,N,2,2).
$$
Let $\varepsilon\in (0, 1/8)$ be a small positive number. In our further argument we specify $M\in\N$ to be the smallest number satisfying the inequality
\be\label{4.1b}
C_0 M^{-2} N \log N \leq \varepsilon,
\ee
with a sufficiently large positive     constant $C_0$.
 It is easily seen that then $|W_{M,N}| \le C(\epsilon) N^{3/2}(\log N)^{1/2}$.
In this subsection, we show that  for each $f\in \mathcal{T}(N)$,
\begin{equation}\label{2-1}
\|f\|_\infty \leq C(\varepsilon) \max_{\bw\in W} |f(\bw)|.
\end{equation}

Assume that  $\bx \in  [0,2\pi)^2$ is such that $\|f\|_\infty =|f(\bx)|$. Let $\ba = (a_1,a_2)$, $a_j\in V_M$, $j=1,2$ be such that $0\le x_j-a_j\leq 2\pi M^{-1}$, $j=1,2$.
Then using Lemma \ref{lem-2-1}, we obtain
$$
\Bl| f(a_1,a_2)-f(a_1,x_2) - f(x_1,a_2)+f(x_1,x_2)\Br|=\Bl| \int_{a_1}^{x_1} \int_{a_2}^{x_2} f^{(1,1)}(u,v) \, dvdu\Br|
$$
$$
\leq C_1 N( \log N) M^{-2}\|f\|_\infty \leq \varepsilon \|f\|_\infty
$$
provided $C_1\le C_0$.
In particular, this implies that
\begin{align*}
\max\Bl\{|f(a_1,a_2)|, |f(a_1,x_2)|,  | f(x_1,a_2)| \Br\}\ge \f {1-\varepsilon} 3 \|f\|_\infty.
\end{align*}
Suppose we have (the other two cases are treated in the same way)
\be\label{4.2n}
|f(a_1,x_2)| \ge \frac{1-\varepsilon}{3} \|f\|_\infty.
\ee
 Then by Theorem \ref{thm-2-1} with $d=1$ we obtain
 $$
 |f(a_1,x_2)| \le \max_{u\in\T} |f(a_1,u)|
 $$
 \be\label{4.3n}
   \le C(1) \max_{w\in W(N,1)}|f(a_1,w)| \le C(1)\max_{\bw\in W }|f(\bw)|.
 \ee
 Inequalities (\ref{4.2n}) and (\ref{4.3n}) imply (\ref{2-1}) with $W=W_{M,N}$, where $M$ satisfies condition (\ref{4.1b}).

\subsection{Step from $d-1$ to $d$ } \label{subsection2:2}
In this subsection we  prove Theorem \ref{thm-2-1}  for all $d\ge 2$.
We use induction on the dimension $d$.
Assume that Theorem \ref{thm-2-1} has been proved for the case of  $d-1$. That is, there exists a set $W(N,d-1)\subset [0, 2\pi)^{d-1}$ of at most $C_{d-1} N^{\al_{d-1}}(\log N)^{\b_{d-1}}$ points such that
$$
\|f\|_\infty \leq C(d-1) \max_{\bw\in W(N,d-1)} |f(\bw)|,\   \   \   \forall f\in \mathcal{T}(N,d-1).
$$
For natural numbers $M$ and $N$ define
$$
V(M,N,d,j) := \{\bx\in \T^d: x_j \in V_M,\, \bx^j \in W(N,d-1)\}, \quad j=1,\dots,d.
$$
Finally, define
$$
W(d):=W_{M,N}(d):=  \cup_{j=1}^d V(M,N,d,j).
$$
Let $\varepsilon\in (0, 1/8)$ be a small positive number. In our further argument we specify $M\in \N$ to be the smallest number satisfying the inequality
\[
C_0(d) M^{-d} N (\log N)^{d-1} \leq \varepsilon,
\]
with a sufficiently large positive     constant $C_0(d)$.
 It is easily seen that then
\be\label{4.6n}
  |W_{M,N}(d)| \leq C_{d-1} M N^{\al_{d-1}}(\log N)^{\b_{d-1}}\leq C(d,\varepsilon) N^{\al_d} (\log N)^{\b_d},
\ee
where
$\al_d=\al_{d-1}+\f 1d =\sum_{j=1}^d \f 1j$ and
$\b_d=\b_{d-1} + 1-\f 1d =d-\al_d$.

Assume that  $\bx \in  [0,2\pi)^d$ is such that $\|f\|_\infty =|f(\bx)|$. Let $\ba = (a_1,\dots,a_d)$, $a_j\in V_M$, $j=1,\dots,d$ be such that $0\le x_j-a_j\leq 2\pi M^{-1}$, $j=1,\dots,d$.
By a straightforward calculation we have
\begin{align*}
& \Bl|\int_{a_1}^{x_1} \dots \int_{a_d}^{x_d}  f^{(1,\dots,  1)}(u_1,  \dots, u_d)\, du_d \dots du_1 \Br|
=\Bl| \sum_{\mathbf{y}\in \mathbf{A}} (-1)^{n_{\mathbf{y}}}  f(\mathbf{y})\Br|,
\end{align*}
where
$$
\mathbf{A}:=\Bl\{(y_1,\dots, y_d):\   \  y_j=a_j\  \ \text{or}\  \ x_j\ \  \text{for $j=1,\dots, d$}\Br\}
$$
and
$$n_{\mathbf{y}} =\left|\Bl\{ j:\  \ y_j =a_j,\   \ 1\leq j\leq d\Br\}\right|.
$$
It follows by Lemma \ref{lem-2-1} that
\begin{align*}
\sum_{\mathbf{y}\in \mathbf{A}\setminus \{\mathbf{x}\} } |f(\mathbf {y})| \ge (1-C(d) N(\log N)^{d-1} M^{-d})\|f\|_\infty \ge (1-\varepsilon)\|f\|_\infty,
\end{align*}
provided $C(d)\le C_0(d)$.
This implies
$$
\max_{\mathbf{y}\in \mathbf{A}\setminus \{\mathbf{x}\}} |f(\mathbf{y})|\ge \f{1-\varepsilon}{2^d-1} \|f\|_\infty.
$$
Let $\by^0\in \mathbf{A}\setminus \{\mathbf{x}\}$ be the one for which the inequality
\be\label{4.4n}
|f(\by^0)| \ge \frac{1-\varepsilon}{2^d-1}\|f\|_\infty
\ee
holds. Then there exists  $j:=j_{\mathbf{y^0}}\in\{1,\dots, d\}$  such that $y_j^0=a_j$. For simplicity of notations assume that $j=1$.  By the induction assumption and the definition of $W_{M,N}(d)$ we have
$$
|f(\mathbf{y^0})|\leq \sup_{\mathbf{u}\in [0, 2\pi)^{d-1}}|f(a_1,  \mathbf{u})|\leq C(d-1)\max_{\mathbf{w}\in W(N,d-1)} |f(a_1, \mathbf{w})|
$$
\be\label{4.5n}
  \leq  C(d-1) \max_{\mathbf{w}\in W_{M,N}(d)}|f(\bw)| .
 \ee
 Combining inequalities (\ref{4.4n}), (\ref{4.5n}) and taking into account bound (\ref{4.6n}) we
 complete the proof of Theorem \ref{thm-2-1} with $W(N,d) = W_{M,N}(d)$.

 \subsection{Some historical remarks and an application to  Remez inequalities}
 It is well known (see Subsection 6.1 for a detailed discussion) that
 $$
 \Tr(\Pi(N)) \in \cM(C(d)N^d,\infty),\quad \Pi(N):= \{\bk\in\Z^d: |k_j| \le N, j=1,\dots,d\}.
 $$
 In particular, this implies that
 $$
 \Tr(N) \in \cM(C(d)N^d,\infty).
 $$
 Theorem \ref{thm-2-1} shows that we can improve the above relation to
 $$
 \Tr(N) \in \cM(C(d)N^{\alpha_d}(\log N)^{\beta_d},\infty).
 $$
 Note that $\alpha_d \asymp \ln d$.
 A trivial lower bound for $m$ in the inclusion $\Tr(N) \in \cM(m,\infty)$ is $m\ge \dim(\Tr(N)) \asymp N(\log N)^{d-1}$. The following nontrivial lower bound was obtained in \cite{KT3} -- \cite{KaTe03}.
 \begin{Theorem}\label{C2.5.1} Let a set $W\subset \mathbb T^2$ have a property:
 $$
\forall t \in \Tr(N) \qquad \|t\|_\infty \le b(\log N)^\al \max_{\bw\in W} |t(\bw)|
$$
with some $0\le \al <1/2$. Then
$$
|W| \ge C_1 N \log N e^{C_2b^{-2}(\log N)^{1-2\al}}.
$$
\end{Theorem}
In particular, Theorem \ref{C2.5.1} with $\alpha=0$ implies that a necessary condition on $m$ for  inclusion $\Tr(N) \in \cM(m,\infty)$ is $m\ge \dim(\Tr(N)) N^c$ with positive absolute constant $c$.

An operator $T_N$ with the following properties was constructed in \cite{T93}.
The operator $T_N$ has the form
$$
T_N(f) = \sum_{j=1}^m f(\bx^j) \psi_j(\bx),\quad m\le c(d)N (\log N)^{d-1},\quad \psi_j \in \Tr(N 2^d)
$$
and
\be\label{3.21n}
T_N(f) =f,\quad f\in \Tr(N),
\ee
\be\label{3.22n}
\|T_N\|_{L_\infty\to L_\infty} \asymp (\log N)^{d-1}.
\ee
Points $\{\bx^j\}$ are form the Smolyak net.
Properties (\ref{3.21n}) and (\ref{3.22n}) imply that all $f\in\Tr(N)$ satisfy the discretization inequality (see \cite{KaTe03})
\[
\|f\|_\infty \le C(d)(\log N)^{d-1} \max_{1\le j\le m} |f(\bx^j)|.
\]

The general form of the Remez inequality for a function $f\in X_N\subset L_p(\Omega)$, $0<p\le \infty$, reads as follows:
 for any Lebesgue measurable $B\subset \Omega$ with the measure $\meas(B)\le b<1$
$$
\|f\|_{L_p(\Omega)} \le C (N, \meas(B), p)\|f\|_{L_p(\Omega\setminus B)}.
$$
Applications of Remez type inequalities include many different results in approximation theory and harmonic analysis; see \cite{remezpaper} for more details and references.

For trigonometric polynomials  $\Tr(Q)$ with frequencies from $Q\subset \Z^d$ (here $X_N=\Tr(Q)$ and $\Omega=\T^d$) the following result is well known \cite{remez}.
For $d\ge 1$ and
$$
Q=\Pi(\bN):= \{\bk \in \Z^d :   |k_j| \le N_j, \quad j=1,\dots,d\},
$$
where $N_j\in\N$, for any $p\in (0,\infty]$,
we have that $C (N, \meas(B), p)= C (d,  p)$ provided that $$\meas(B)\le \frac{C}{\prod_{j=1}^dN_j}.$$

The investigation of the Remez-type inequalities for the
hyperbolic cross trigonometric polynomials
with
\begin{equation}\label{qq} Q =\Gamma(N)=\Bl\{\mathbf{k}\in\Z^d:\  \  \   \prod_{j=1}^d \max\{ |k_j|, 1\} \leq N\Br\}
\end{equation}
 has been recently initiated in \cite{remezpaper}.
 It turns out that for such  polynomials
the problem to obtain the optimal Remez inequalities
has different solutions when $p<\infty$ and  $p=\infty$.
If $p<\infty$, then  $C (N, \meas(B), p)= C (d,  p)$ provided that $$\meas(B)\le \frac{C}{N}.$$

The case   $p=\infty$  was also studied in \cite{remezpaper}.
 \begin{Theorem}\label{T3.1}
  There exist two positive constants $C_1(d)$ and $C_2(d)$ such that for any set $B\subset \T^d$ of normalized measure $$\meas(B)\le \frac{C_2(d)}{N(\log N)^{d-1}}$$ and for any
$f\in \Tr(Q)$, where $Q$ is given by (\ref{qq}), we have
\begin{equation}\label{qqq}\|f\|_\infty \le C_1(d)(\log N)^{d-1} \sup_{{\mathbf u}\in \T^d \setminus B} |f({\mathbf u})|.
\end{equation}\end{Theorem}
It is worth mentioning that this result is sharp with respect to the logarithmic factor.
 This is because
the following statement is false (see \cite{remezpaper}).

{\it There exist  $\de>0$, $A$, $c$, and $C$ such that for any $f\in \Tr(N)$ and any set $B\subset \T^d$ of measure
$\meas(B) \le (cN(\log N)^A)^{-1}$ the Remez-type inequality holds
$$\|f\|_\infty \le C(\log N)^{(d-1)(1-\de)} \sup_{{\mathbf u}\in \T^d\setminus B} |f({\mathbf u})|.
$$}

Let us now give a nontrivial Remez inequality with no logarithmic factor in (\ref{qqq}).
It follows from the fact  \cite[Th.2.4]{remezpaper} that
the discretization inequality implies Remez inequality in $L_\infty$.
Together with Theorem \ref{thm-2-1} this implies

\begin{Theorem}\label{thm-remez}
Let $d\ge 2$, $\al_d=\sum_{j=1}^d \f 1j$, and
$\b_d=d-\al_d$.
There exist two positive constants $C_1(d)$ and $C_2(d)$ such that for any set $B\subset \T^d$ of normalized measure $$\meas(B)\le \frac{C_2(d)}{N^{\alpha_d}(\log N)^{\beta_d}}$$ and for any
$f\in \Tr(Q)$, where $Q$ is given by (\ref{qq}), we have
$$\|f\|_\infty \le C_1(d) \sup_{{\mathbf u}\in \T^d \setminus B} |f({\mathbf u})|.
$$
\end{Theorem}

\section{Marcinkiewicz-type inequality for general trigonometric polynomials for $q=\infty$}
\label{ungen}

In this section we present results from \cite{KT168} and \cite{KT169}. More specifically, we
demonstrate how to obtain the first part of Theorem \ref{ITmain}.

\subsection{Small ball inequality}
\label{A}

In this section we consider special subspaces of the univariate trigonometric polynomials. For $n\in \N$ let $\K :=\{k_j\}_{j=n}^{2n-1}$ be a finite set of $n$ natural numbers such that $k_{j+1}>k_j$, $j=n,\dots,2n-2$. For $\nu\in\N$ define
$$
\Tr(\K,\nu) := \left\{f\, :\, f=\sum_{j=n}^{2n-1} p_j(x) e^{ik_jx}\right\},
$$
where $p_j\in \Tr(\nu):=\Tr([-\nu,\nu])$, $j=n,\dots, 2n-1$, $n=1,2,\dots$.

We prove some results for a set $\K$ and a number $\nu$ satisfying the following condition.

{\bf Condition L.} Suppose that all $k_j$, $j=n,\dots,2n-1$, are divisible by $k_n$ and that there exists a number $b>1$ such that $k_{j+1}\ge bk_j$, $j=n,\dots,2n-2$. Moreover, there is a constant $K$ such that we have $\nu\le (b-1)k_n/3$ and $\nu n \le Kk_n$.

\begin{Theorem}\label{AT1} Suppose that the pair $\K$, $\nu$ satisfies Condition L. Then there exists a constant $C=C(K,b)$, which may only depend on $K$ and $b$ such that for any
\be\label{A1}
f=\sum_{j=n}^{2n-1} p_j(x) e^{ik_jx}
\ee
we have for all $x\in [0,2\pi)$
\be\label{A2}
\sum_{j=n}^{2n-1} |p_j(x)| \le C\|f\|_\infty.
\ee
\end{Theorem}
\begin{proof} Take a point $x_0\in [0,2\pi)$ and prove (\ref{A2}) for this point. First of all,
considering a convolution of $f(x)$ with $\V_{\nu}(x) e^{ik_jx}$, where $\V_N(x)$ is the
de la Vall{\'e}e Poussin kernel (see \cite{VTbookMA}, p. 10), we obtain
\be\label{A3}
\|p_j\|_\infty \le C_1\|f\|_\infty =: A,\quad j=n,\dots, 2n-1.
\ee
Second, consider $f(x)$ with $x=x_0+ y/k_n$, $y\in [0,2\pi)$. By the Bernstein inequality
we get from (\ref{A3}) for $y\in [0,2\pi)$
\be\label{A4}
|p_j(x_0+y/k_n)-p_j(x_0)| \le A\nu 2\pi/k_n.
\ee
We now use a well known fact from the theory of lacunary series (see \cite{Z}, Ch.6). For a function
\be\label{A5}
g(y):=\sum_{j=n}^{2n-1} c_je^{ik_jy}
\ee
we have
\be\label{A6}
\sum_{j=n}^{2n-1} |c_j| \le C_2(b)\|g\|_\infty
\ee
with a constant  $C_2(b)$, which may only depend on $b$.

Consider a function
$$
g(y) := \sum_{j=n}^{2n-1} p_j(x_0) e^{i(k_jx_0+k_jy/k_n)}.
$$
Then, $\{k_j/k_n\}_{j=n}^{2n-1}$ is a lacunary set and (\ref{A4}), (\ref{A6}) imply
$$
\sum_{j=n}^{2n-1} |p_j(x_0)| \le C_2(b)\|g\|_\infty \le \max_{y\in[0,2\pi)} |f(x_0+y/k_n)| +
nA\nu 2\pi/k_n \le C(K,b)\|f\|_\infty.
$$
This completes the proof of Theorem \ref{AT1}.

\end{proof}
Theorem \ref{AT1} implies immediately the following result.
\begin{Theorem}\label{AT2} Suppose that the pair $\K$, $\nu$ satisfies Condition L. Then there exists a constant $C=C(K,b)$, which may only depend on $K$ and $b$ such that for any
\be\label{A1'}
f=\sum_{j=n}^{2n-1} p_j(x) e^{ik_jx}
\ee
we have
\be\label{A2'}
\sum_{j=n}^{2n-1} \|p_j\|_1 \le C\|f\|_\infty.
\ee
\end{Theorem}

 As above
  for a finite set $\L \subset \Z^d$ denote $\Tr(\La)$ the set of trigonometric polynomials with frequencies in $\La$. Denote
$$
  \Tr(\L)_p :=\{f\in\Tr(\L): \|f\|_p\le 1\}.
$$
For a finite set $\L$ we assign to each $f = \sum_{\bk\in \L} \hat f(\bk) e^{i(\bk,\bx)}\in \Tr(\L)$ a vector
$$
A(f) := \{(\text{Re}(\hat f(\bk)), \text{Im}(\hat f(\bk))),\quad \bk\in \L\} \in \R^{2|\L|}
$$
where $|\L|$ denotes the cardinality of $\L$ and define
$$
B_\L(L_p) := \{A(f) : f\in \Tr(\L)_p\}.
$$
The volume estimates of the sets $B_\L(L_p)$ and related questions have been studied in a number of papers: the case $\L=[-n,n]$, $p=\infty$ in \cite{K1}; the case $\L=[-N_1,N_1]\times\cdots\times[-N_d,N_d]$, $p=\infty$ in \cite{TE2}, \cite{T3}.   In the case $\L = \Pi(\bN,d) := [-N_1,N_1]\times\cdots\times[-N_d,N_d]$, $\bN:=(N_1,\dots,N_d)$, the following estimates follow from results of \cite{K1}, \cite{TE2}, and \cite{T3} (see also \cite{VTbookMA}, p.333).
\begin{Theorem}\label{AT3} For any $1\le p\le \infty$ we have
$$
(vol(B_{\Pi(\bN,d)}(L_p)))^{(2|\Pi(\bN,d)|)^{-1}} \asymp |\Pi(\bN,d)|^{-1/2},
$$
with constants in $\asymp$ that may depend only on $d$.
\end{Theorem}

Denote
$$
\La(\K,\nu):= \cup_{j=n}^{2n-1}\La_j(\nu),\qquad \La_j(\nu):= [k_j-\nu,k_j+\nu].
$$
 We now estimate from above
the $vol(B_{\La(\K,\nu)}(L_\infty))$ under Condition L.
\begin{Theorem}\label{AT4} Suppose that the pair $\K$, $\nu$ satisfies Condition L. Then
$$
(vol(B_{\La(\K,\nu)}(L_\infty)))^{(2|\La(\K,\nu)|)^{-1}} \le C'(n|\La(\K,\nu)|)^{-1/2}.
$$

\end{Theorem}
\begin{proof} Let $f\in \Tr(\La(\K,\nu))$ and $\|f\|_\infty \le 1$. Then $f$ has a form (\ref{A1'}) and by Theorem \ref{AT2} we get
\be\label{A9}
\sum_{j=n}^{2n-1} \|p_j\|_1 \le C.
\ee
Inequality (\ref{A9}) guarantees that there exist numbers $a_j:= [n\|p_j\|_1/C]+1\in \N$ such that
\be\label{A10}
\|p_j\|_1 \le \frac{Ca_j}{n},\qquad \sum_{j=n}^{2n-1} a_j \le 2n.
\ee
Denote $A(n):= \{\ba=(a_n,\dots,a_{2n-1})\in \N^n: a_n+\dots+a_{2n-1}\le 2n\}$. Then
\be\label{A11}
vol(B_{\La(\K,\nu)}(L_\infty)) \le \sum_{\ba \in A(n)} \prod_{j=n}^{2n-1} vol(B_{\La_j(\nu)}(L_1))(Ca_j/n)^{2(2\nu+1)}.
\ee
For $\{a_j\}$ satisfying inequality (\ref{A10}) we obtain
\be\label{A12}
a_n\cdots a_{2n-1} \le ((a_n+\cdots+a_{2n-1})/n)^n \le 2^n.
\ee
It is known that
$$
|\{(b_1,\dots,b_n)\in \Z_+^n: b_1+\cdots+b_n =q\}| = \binom{n+q-1}{q}.
$$
Therefore, for the number of summands in (\ref{A11}) we have
\be\label{A13}
|A(n)| \le \sum_{q=0}^n \binom{n+q-1}{q} \le 2^{2n}.
\ee
Combining (\ref{A11}) -- (\ref{A13}), using Theorem \ref{AT3} and taking into account that $|\La(\K,\nu)|= n(2\nu+1)$ we obtain
\be\label{A14}
(vol(B_{\La(\K,\nu)}(L_\infty)))^{(2|\La(\K,\nu)|)^{-1}} \le C'(n|\La(\K,\nu)|)^{-1/2}.
\ee

\end{proof}

\subsection{Discretization}
\label{B}

The above Theorem \ref{AT4} implies an interesting and surprising result on discretization for polynomials from $\Tr(\La(\K,\nu))$. We derive from Theorem \ref{BT3}, which is a corollary of Theorem \ref{AT4},
that there is no analog of the Marcinkiewicz theorem  in $L_\infty$ for polynomials from $\Tr(\La(\K,\nu))$.
We present here some results from \cite{KaTe03} (see also \cite{VTbookMA}, pp. 344--345).
We begin with the following conditional statement.
\begin{Theorem}\label{BT1} Assume that a finite set $\La\subset \Z^d$ has the following properties:
\begin{equation}\label{B1}
(vol(B_\La(L_\infty)))^{1/D} \le K_1D^{-1/2},\quad D:=2|\L|,
\end{equation}
and a set $\Omega_M =\{\bx^1,\dots,\bx^M\}$ satisfies the condition
\begin{equation}\label{B2}
\forall f \in \Tr(\L) \qquad \|f\|_\infty\le K_2\|f\|_{\Omega_M},\quad
\|f\|_{\Omega_M}:=\max_{\bx\in \Omega_M}|f(\bx)|.
\end{equation}
Then there exists an absolute constant $c>0$ such that
$$
M\ge \frac{|\La|}{e}e^{c(K_1K_2)^{-2}}.
$$
\end{Theorem}
\begin{proof} We use the following result of E. Gluskin \cite{G}.
\begin{Theorem}\label{BT2} Let $Y=\{\by_1,\dots,\by_S\} \subset \R^D$, $\|\by_i\|=1$, $i=1,\dots,S$, $S\ge D$, and
$$
W(Y) := \{\bx\in \R^D:|(\bx,\by_i)| \le 1,\quad i=1,\dots,S\}.
$$
Then
$$
(vol(W(Y)))^{1/D} \ge C(1+\ln (S/D))^{-1/2}.
$$
\end{Theorem}

By our assumption (\ref{B2}) we have
\be\label{B1p}
\forall f \in \Tr(\La) \qquad \|f\|_\infty\le K_2\|f\|_{\Omega_M}.
\ee
Thus,
\be\label{B2p}
\{A(f):f\in \Tr(\L),\quad |f(\bx)|\le K_2^{-1},\quad \bx\in \Omega_M\} \subseteq B_\L(L_\infty).
\ee
Further
$$
|f(\bx)|^2 = |\sum_{\bk\in\L}\hat f(\bk) e^{i(\bk,\bx)}|^2 =
$$
$$
\left(\sum_{\bk\in\L}\text{Re}\hat f(\bk) \cos(\bk,\bx) - \text{Im}\hat f(\bk) \sin(\bk,\bx)\right)^2
$$
$$
+\left(\sum_{\bk\in\L}\text{Re}\hat f(\bk) \sin(\bk,\bx) + \text{Im}\hat f(\bk) \cos(\bk,\bx)\right)^2 .
$$
We associate with each point $\bx\in \Omega_M$ two vectors $\by^1(\bx)$ and $\by^2(\bx)$ from $\R^D$:
$$
\by^1(\bx) := \{(\cos(\bk,\bx),-\sin(\bk,\bx)),\quad \bk\in \L\},
$$
$$
 \by^2(\bx) := \{(\sin(\bk,\bx),\cos(\bk,\bx)),\quad \bk\in \L\}.
$$
Then
$$
\|\by^1(\bx)\|^2 =\|\by^2(\bx)\|^2 = |\L|
$$
and
$$
|f(\bx)|^2 = (A(f),\by^1(\bx))^2 +(A(f),\by^2(\bx))^2.
$$
It is clear that the condition $|f(\bx)| \le K_2^{-1}$ is satisfied if
$$
|(A(f),\by^i(\bx))| \le 2^{-1/2}K_2^{-1}, \quad i=1,2.
$$
Let now
$$
Y:=\{\by^i(\bx)/\|\by^i(\bx)\|:\quad \bx\in \Omega_M,\quad i=1,2\}.
$$
Then $S=2M$ and by Theorem \ref{BT2}
\be\label{B3p}
(vol(W(Y)))^{1/D} \ge C(1+\ln (S/D))^{-1/2} .
\ee
Using that the condition
$$
|(A(f),\by^i(\bx))|\le 1
$$
is equivalent to the condition
$$
|(A(f),\by^i(\bx)/\|\by^i(\bx)\|)| \le (D/2)^{-1/2}
$$
we get from (\ref{B2p}) and (\ref{B3p})
$$
(vol(B_\L(L_\infty)))^{1/D} \ge C'D^{-1/2}K_2^{-1}(1+\ln (S/D))^{-1/2}.
$$
We now use our assumption (\ref{B1}) and obtain
\be\label{B4p}
K_1K_2 \ge C'(\ln (eM/|\La|))^{-1/2}.
\ee
This completes the proof of Theorem \ref{BT1}.
\end{proof}

 We now give some corollaries of Theorem \ref{BT1}.
\begin{Theorem}\label{BT3} Assume that a finite set $\Omega\subset \mathbb T$ has
the following property.
\begin{equation}\label{B4}
\forall t\in \Tr(\La(\K,\nu)) \qquad \|t\|_\infty \le K_2\|t\|_{\Omega}.
\end{equation}
Then
$$
|\Omega| \ge \frac{|\La(\K,\nu)|}{e}e^{Cn/K_2^2}
$$
with an absolute constant $C>0$.
\end{Theorem}
\begin{proof} By Theorem \ref{AT2} we have with $D:=2|\La(\K,\nu)|$
$$
(vol(B_{\La(\K,\nu)}(L_\infty)))^{1/D} \le C(n|\La(\K,\nu)|)^{-1/2} \le Cn^{-1/2}D^{-1/2}
$$
with a constant $C>0$, which may depend on $K$ and $b$. Using Theorem \ref{BT1} we obtain
$$
|\Omega|\ge \frac{|\La(\K,\nu)|}{e}e^{Cn/K_2^2}.
$$
This proves Theorem \ref{BT3}.
 \end{proof}
\begin{Corollary}\label{BC1} Denote $N:= |\La(\K,\nu)|$. Theorem \ref{BT3} implies for $m\ge N$
\be\label{B5}
D(\La(\K,\nu),m) \ge Cn^{1/2} \left(\ln\frac{em}{N}\right)^{-1/2}.
\ee
In particular, (\ref{B5}) with $\nu=0$ implies for $m\le c'N$ that
\be\label{B6}
D(\La(\K,0),m) \ge C'N^{1/2} \quad \Rightarrow\quad D(N,m)\ge C'N^{1/2}.
\ee
\end{Corollary}

\begin{Remark}\label{BR1} In a particular case $K_2=Bn^\al$, $0\le \al\le 1/2$, Theorem \ref{BT3} gives
$$
|\Omega|\ge \frac{|\La(\K,\nu)|}{e}e^{CB^{-2}n^{1-2\al}}.
$$
\end{Remark}
\begin{Corollary}\label{BC2} Let a set $\Omega\subset \mathbb T$ have a property:
 $$
\forall t \in \Tr(\La(\K,\nu)) \qquad \|t\|_\infty \le Bn^\al\|t\|_{\infty,\Omega}
$$
with some $0\le \al <1/2$. Then
$$
|\Omega| \ge C_3|\La(\K,\nu)|e^{CB^{-2}n^{1-2\al}} \ge C_1(K,b,B,\al)|\La(\K,\nu)|e^{C_2(K,b,B,\al)n^{1-2\al}}.
$$
\end{Corollary}

In the case $\nu=0$  we have $|\La(\K,\nu)|=n$ and, therefore, Corollary \ref{BC2} with $\al=0$
claims that for $D(\La(\K,0),m)\le B$ we need $m\ge Ce^{cn}$ points for discretization.

 \section{Universal discretization}
\label{ud}

{\bf Universal discretization problem.} This problem is about finding (proving existence) of
a set of points, which is good in the sense of the above Marcinkiewicz-type discretization
for a collection of linear subspaces. We formulate it in an explicit form. Let $\cX_N:= \{X_{N,j}\}_{j=1}^k$ be a collection of linear subspaces $X_{N,j}$ of the $L_q(\Omega)$, $1\le q \le \infty$. We say that a set $\{\xi^\nu \in \Omega, \nu=1,\dots,m\}$ provides {\it universal discretization} for the collection $\cX_N$ if, in the case $1\le q<\infty$, there are two positive constants $C_i(d,q)$, $i=1,2$, such that for each $j\in\{1,\dots,k\}$ and any $f\in X_{N,j}$ we have
\[
C_1(d,q)\|f\|_q^q \le \frac{1}{m} \sum_{\nu=1}^m |f(\xi^\nu)|^q \le C_2(d,q)\|f\|_q^q.
\]
In the case $q=\infty$  for each $j\in\{1,\dots,k\}$ and any $f\in X_{N,j}$ we have
\be\label{1.2u}
C_1(d)\|f\|_\infty \le \max_{1\le\nu\le m} |f(\xi^\nu)| \le  \|f\|_\infty.
\ee

\subsection{Anisotropic trigonometric polynomials}
\label{udsub1}
The problem of universal discretization for some special subspaces of the
trigonometric polynomials was studied in \cite{VT160}. Recall that for   a finite subset $Q$ of $\Z^d$,
$$
\Tr(Q):= \Bigl\{f: f=\sum_{\bk\in Q}c_\bk e^{i(\bk,\bx)},\   \  c_{\mathbf{k}}\in \mathbb{C},\  \  \bk\in Q\Bigr\}.
$$
For $\bs\in\Z^d_+$
define
$$
R(\bs) := \{\bk \in \Z^d :   |k_j| < 2^{s_j}, \quad j=1,\dots,d\}.
$$
Clearly, $R(\bs) = \Pi(\bN)$ with $N_j = 2^{s_j}-1$. Consider the collection $\cC(n,d):= \{\Tr(R(\bs)), \|\bs\|_1=n\}$.
The following result was proved in \cite{VT160}.
\begin{Theorem}\label{udT1}   For every $1\le q\le\infty$ there exists a large enough constant $C(d,q)$, which depends only on $d$ and $q$, such that for any $n\in \N$ there is a set $\xi:=\{\xi^\nu\}_{\nu=1}^m\subset \T^d$, with $m\le C(d,q)2^n$ that provides universal discretization in $L_q$   for the collection $\cC(n,d)$.
\end{Theorem}
Theorem \ref{udT1}, basically, solves the universal discretization problem for the collection $\cC(n,d)$. It provides the upper bound  $m\le C(d,q)2^n$ with $2^n$ being of the order of the dimension of each $\Tr(R(\bs))$ from the collection $\cC(n,d)$.
 Obviously, the lower bound for the cardinality of a set, providing the Marcinkiewicz discretization theorem for $\Tr(R(\bs))$ with $\|\bs\|_1=n$, is $\ge C(d)2^n$.

It was observed in \cite{VT160} that the universal discretization problem in $L_\infty$
for the collection $\cC(n,d)$ is, in a certain sense, equivalent to the minimal {\it dispersion} problem. Let us describe this phenomenon in detail. Let $d\ge 2$ and $[0,1)^d$ be the $d$-dimensional unit cube. For $\bx,\by \in [0,1)^d$ with $\bx=(x_1,\dots,x_d)$ and $\by=(y_1,\dots,y_d)$ we write $\bx < \by$ if this inequality holds coordinate-wise. For $\bx<\by$ we write $[\bx,\by)$ for the axis-parallel box $[x_1,y_1)\times\cdots\times[x_d,y_d)$ and define
$$
\cB:= \{[\bx,\by): \bx,\by\in [0,1)^d, \bx<\by\}.
$$
For $n\ge 1$ let $T$ be a set of points in $[0,1)^d$ of cardinality $|T|=n$. The volume of the largest empty (from points of $T$) axis-parallel box, which can be inscribed in $[0,1)^d$, is called the dispersion of $T$:
$$
\text{disp}(T) := \sup_{B\in\cB: B\cap T =\emptyset} vol(B).
$$
An interesting extremal problem is to find (estimate) the minimal dispersion of point sets of fixed cardinality:
$$
\text{disp*}(n,d) := \inf_{T\subset [0,1)^d, |T|=n} \text{disp}(T).
$$
It is known that
\be\label{ud1}
\text{disp*}(n,d) \le C^*(d)/n.
\ee
Inequality (\ref{ud1}) with $C^*(d)=2^{d-1}\prod_{i=1}^{d-1}p_i$, where $p_i$ denotes the $i$th prime number, was proved in \cite{DJ} (see also \cite{RT}). The authors of \cite{DJ} used the Halton-Hammersly set of $n$ points (see \cite{Mat}). Inequality (\ref{ud1}) with $C^*(d)=2^{7d+1}$ was proved in
\cite{AHR}. The authors of \cite{AHR}, following G. Larcher, used the $(t,r,d)$-nets.
\begin{Definition}\label{udD1} A $(t,r,d)$-net (in base $2$) is a set $T$ of $2^r$ points in
$[0,1)^d$ such that each dyadic box $[(a_1-1)2^{-s_1},a_12^{-s_1})\times\cdots\times[(a_d-1)2^{-s_d},a_d2^{-s_d})$, $1\le a_j\le 2^{s_j}$, $j=1,\dots,d$, of volume $2^{t-r}$ contains exactly $2^t$ points of $T$.
\end{Definition}
A construction of such nets for all $d$ and $t\ge Cd$, $r\ge t$  is given in \cite{NX}.
The following conditional theorem, based on the concept of dispersion, was proved in \cite{VT160}.

\begin{Theorem}\label{udT2} Let a set $T$ with cardinality $|T|= 2^r=:m$ have dispersion
satisfying the bound disp$(T) < C(d)2^{-r}$ with some constant $C(d)$. Then there exists
a constant $c(d)\in \N$ such that the set $2\pi T:=\{2\pi\bx: \bx\in T\}$ provides the universal discretization in $L_\infty$ for the collection $\cC(n,d)$ with $n=r-c(d)$.
\end{Theorem}

The following Theorem \ref{udT3} (see \cite{VT160}) can be seen as an inverse to Theorem~\ref{udT2}.

 \begin{Theorem}\label{udT3}  Assume that $T\subset [0,1)^d$ is such that the set $2\pi T$ provides universal discretization in $L_\infty$ for the collection
 $\cC(n,d)$ with a constant $C_1(d)$ (see (\ref{1.2u})). Then there exists a positive constant $C(d)$ such that  disp$(T) \le C(d)2^{-n}$.
 \end{Theorem}

 \subsection{Arbitrary trigonometric polynomials}
 \label{udsub2}
 For $n\in \N$ denote $\Pi_n :=\Pi(\bN)\cap \Z^d$ with $\bN =(2^{n-1}-1,\dots,2^{n-1}-1)$, where, as above, $\Pi(\bN) := [-N_1,N_1]\times\cdots\times[-N_d,N_d]$.
 Then $|\Pi_n| = (2^n-1)^d <2^{dn}$. Let $v\in\N$ and $v\le |\Pi_n|$. Consider
 $$
 \cS(v,n):= \{Q\subset \Pi_n : |Q|=v\}.
 $$
 Then it is easy to see that
 \[
 |\cS(v,n)| =\binom{|\Pi_n|}{v}<2^{dnv}.
 \]

 We are interested in solving the following problem of universal discretization.
 For a given $\cS(v,n)$ and $q\in [1,\infty)$ find a condition on $m$ such that there exists a set
 $\xi = \{\xi^\nu\}_{\nu=1}^m$ with the property: for any $Q\in \cS(v,n)$ and each
 $f\in \Tr(Q)$ we have
 \[
 C_1(q,d)\|f\|_q^q \le \frac{1}{m}\sum_{\nu=1}^m |f(\xi^\nu)|^q \le C_2(q,d)\|f\|^q_q.
 \]
 We present results for $q=2$ and $q=1$.

 {\bf The case $q=2$.} We begin with a general construction. Let $X_N = \sp(u_1,\dots,u_N)$, where $\{u_j\}_{j=1}^N$ is a real orthonormal system on $\T^d$.
With each $\bx\in\T^d$ we associate the matrix $G(\bx) := [u_i(\bx)u_j(\bx)]_{i,j=1}^N$. Clearly, $G(\bx)$ is a symmetric matrix. For a set of points $\xi^k\in \T^d$, $k=1,\dots,m$, and $f=\sum_{i=1}^N b_iu_i$ we have
$$
\frac{1}{m}\sum_{k=1}^m f(\xi^k)^2 - \int_{\mathbb{T}^d}  f(x)^2 d\mu = {\mathbf b}^T\left(\frac{1}{m}\sum_{k=1}^m G(\xi^k)-I\right){\mathbf b},
$$
where ${\mathbf b} = (b_1,\dots,b_N)^T$ is the column vector. Therefore,
\[
\left|\frac{1}{m}\sum_{k=1}^m f(\xi^k)^2 - \int_{\mathbb{T}^d} f(x)^2 d\mu \right| \le
\left\|\frac{1}{m}\sum_{k=1}^m G(\xi^k)-I\right\|\|{\mathbf b}\|_2^2.
\]
We recall that
the system $\{u_j\}_{j=1}^N$ satisfies
{Condition~{\bf E}} (see~\eqref{ud5}) if there exists a constant $t$ such that
$$
w(x):=\sum_{i=1}^N u_i(x)^2 \le Nt^2.
$$
Let points $\bx^k$, $k=1,\dots,m$, be independent uniformly distributed on $\T^d$ random variables. Then with a help of deep results on random matrices (see Theorem~\ref{T5.3} or~\cite[Theorem 1.1]{Tro12}) it was proved in~\cite{VT159} that
\[
\bP\left\{\left\|\sum_{k=1}^m (G(\mathbf{x}^k)-I) \right\|\ge m\eta\right\} \le N\exp\left(-\frac{m\eta^2}{ct^2N}\right)
\]
with an absolute constant $c$.
Consider real trigonometric polynomials from the collection $\cS(v,n)$. Using the union bound for the probability we get that the probability of the event
$$
\left\|\sum_{k=1}^m (G_Q(\mathbf{x}^k)-I) \right\|\le m\eta \quad\text{for all}\quad Q\in \cS(v,n)
$$
is bounded from below by
$$
1- |\cS(v,n)|v\exp\left(-\frac{m\eta^2}{cv}\right).
$$
For any fixed $\eta\in(0,1/2]$ the above number is positive provided $m \ge C(d)\eta^{-2}v^2n$ with large enough $C(d)$.
The above argument proves the following result.

\begin{Theorem}\label{udT4} There exist three positive constants $C_i(d)$, $i=1,2,3$,
such that for any $n,v\in\N$ and $v\le |\Pi_n|$ there is a set $\xi =\{\xi^\nu\}_{\nu=1}^m \subset \T^d$, with $m\le C_1(d)v^2n$, which provides universal discretization
in $L_2$ for the collection $\cS(v,n)$: for any $f\in \cup_{Q\in \cS(v,n)} \Tr(Q)$
$$
C_2(d)\|f\|_2^2 \le \frac{1}{m} \sum_{\nu=1}^m |f(\xi^\nu)|^2 \le C_3(d)\|f\|_2^2.
$$
\end{Theorem}

The classical Marcinkiewicz-type result for $\Tr(\Pi_n)$ provides a universal set $\xi$ with cardinality $m\le C(d)2^{dn}$. Thus, Theorem \ref{udT4} gives a non-trivial result
for $v$ satisfying $v^2n\le C(d)2^{dn}$.

{\bf Case $q=1$.} Similar to the case $q=2$ a result on the universal discretization
for the collection $\cS(v,n)$ will be derived from the probabilistic result on the Marcinkiewicz-type theorem for $\Tr(Q)$, $Q\subset  \Pi_n$. However, the probabilistic technique used in the case of $q=1$ is different from the probabilistic technique used in the case $q=2$. The proof of Theorem 3.1 from \cite{VT159} gives the following result.

\begin{Theorem}\label{udT5} Let points $\bx^j\in\T^d$, $j=1,\dots,m$, be independently and uniformly distributed on $\T^d$. There exist positive constants $C_1(d)$, $C_2$, $C_3$, and $\kappa\in (0,1)$ such that for any $Q\subset \Pi_n$ and $m \ge yC_1(d)|Q|n^{7/2}$, $y\ge 1$,
$$
\bP\left\{\text{For any}\quad f\in\Tr(Q), \quad C_2\|f\|_1 \le \frac{1}{m}\sum_{j=1}^m |f(\bx^j)| \le C_3\|f\|_1\right\} \ge 1-\kappa^y.
$$
\end{Theorem}

Therefore, using the union bound for probability we obtain the Marcinkiewicz-type inequalities
for all $Q\in \cS(v,n)$ with probability at least $1-|\cS(v,n)|\kappa^y$. Choosing $y= y(v,n):= C(d)vn$ with large enough $C(d)$ we get
$$
1-|\cS(v,n)|\kappa^{y(v,n)}>0.
$$
This argument implies the following result on universality in $L_1$.

\begin{Theorem}\label{udT6} There exist three positive constants $C_1(d)$, $C_2$, $C_3$,
such that for any $n,v\in\N$ and $v\le |\Pi_n|$ there is a set $\xi =\{\xi^\nu\}_{\nu=1}^m \subset \T^d$, with $m\le C_1(d)v^2n^{9/2}$, which provides universal discretization
in $L_1$ for the collection $\cS(v,n)$: for any $f\in \cup_{Q\in \cS(v,n)} \Tr(Q)$
$$
C_2\|f\|_1 \le \frac{1}{m} \sum_{\nu=1}^m |f(\xi^\nu)| \le C_3\|f\|_1.
$$
\end{Theorem}

The classical Marcinkiewicz-type result for $\Tr(\Pi_n)$ provides a universal set $\xi$ with cardinality $m\le C(d)2^{dn}$. Thus, Theorem \ref{udT6} gives a non-trivial result
for $v$ satisfying $v^2n^{9/2}\le C(d)2^{dn}$.

\section{Open problems}
\label{OP}

We collect a number of open problems in this section. Probably, some of them are
rather simple and others are very difficult. By listing these problems we want to illustrate
that there are many interesting and important directions to go.

\subsection{Exact}

Results of Section \ref{Ex} solve the problem of exact weighted discretization.
The problem of exact discretization, that is the problem with equal weights $1/m$,
is still open.

\begin{description}
  \item[Open problem 1.]Find necessary and sufficient conditions on $X_N$ for
$X_N \in \cM(c(d)N^2,2,0)$.
\end{description}

Theorem 4.4 from \cite{VT158} gives the following relation for the trigonometric polynomials $\Tr(Q)$ with frequencies from $Q\subset \Z^d$, satisfying some extra conditions,
\be\label{2.2}
\Tr(Q) \in \cM(c(d)|Q|^2,2,0).
\ee
Results of Subsection \ref{Ex--} show that (\ref{2.2}) cannot be improved by replacing $|Q|^2$ by a slower growing function on $|Q|$.

\begin{description}
  \item[Open problem 2.]
Does (\ref{2.2}) hold for all $Q$?
\end{description}

\begin{description}
  \item[Open problem 3 (conjecture).]
For a real subspace $X_N\subset L_2(\Omega,\mu)$ define
$$
m(X_N,w) := \min\{m: \, X_N \in \cM^w(m,2,0)\}.
$$
Let $m=m(X_N,w)$ and let $\{\xi^\nu\}$, $\{\la_\nu\}$, $\nu=1,\dots,m$, be such that for any $f\in X_N$ we have
$$
\int_{\Omega}f^2d\mu = \sum_{\nu=1}^m\la_\nu f(\xi^\nu)^2.
$$
Then $\la_\nu>0$, $\nu=1,\dots,m$.
\end{description}

\subsection{$\cM(m,q)$, $1\le q\le \infty$}
\label{M}

For the trigonometric polynomials the problem is basically solved in the case $q=2$ (see Theorem \ref{NOUth} above and Theorem 1.1 from \cite{VT158}):
\be\label{3.1}
\Tr(Q) \in \cM(c(d)|Q|,2).
\ee

\begin{description}
  \item[Open problem 4.]
Find conditions (necessary and sufficient) for $\Tr(Q) \in \cM(m,q)$ in the case $q\in [1,\infty]\setminus 2$.
\end{description}

Here is a particular case of open problem 4, which is of special interest.

\begin{description}
  \item[Open problem 5.]
Prove open problem 4 for $\Tr(Q_n)$ -- the set of trigonometric polynomials with frequencies from a step hyperbolic cross $Q_n$.
\end{description}

A very interesting and very difficult problem is an analog of open problem~4 for general
subspaces $X_N$:

\begin{description}
  \item[Open problem 6.]
Find conditions (necessary and sufficient) for $X_N \in \cM(m,q)$ in the case $q\in [1,\infty]$. This problem includes conditions on both $X_N$ and $m$.
\end{description}

All the above problems, especially in the case of general $X_N$, are of interest for
$\cM^w(m,q)$. Open problem 6 contains interesting subproblems. We discuss some of them.

\begin{description}
  \item[Open problem 6a.]
 Let $\Omega:=[0,1]^d$ be a unit $d$-dimensional cube and $\mu$ be a probability measure on $\Omega$. Take $q\in [1,\infty)$. Is the following statement true? There exists $C(d,q)$ such that for any $N$-dimensional
 subspace $X_N \subset L_q(\Omega,\mu)$ we have $X_N\in \cM(C(d,q)N,q)$.
 \end{description}

 \begin{description}
  \item[Open problem 6b.]
 Let $\Omega:=[0,1]^d$ be a unit $d$-dimensional cube and $\mu$ be a probability measure on $\Omega$. Take $q\in [1,\infty)$. Is the following statement true? There exists $C(d,q)$ such that for any $N$-dimensional
 subspace $X_N \subset L_q(\Omega,\mu)$ we have $X_N\in \cM^w(C(d,q)N,q)$.
 \end{description}

It turns out that results for the Marcinkiewicz discretization problems in $L_q$, $1\le q<\infty$ and in $L_\infty$ 
are different. We demonstrate this phenomenon on the above Open problems~6a and~6b. In analogy with 
Open problems~6a and~6b one could formulate the following version of them in the case of $L_\infty$. 

 \begin{description}
  \item[Open problem 6c.]
 Let $\Omega:=[0,1]^d$ be a unit $d$-dimensional cube.  Is the following statement true? There exists $C(d)$ such that for any $N$-dimensional
 subspace $X_N \subset L_\infty(\Omega)$ of continuous functions we have $X_N\in \cM(C(d)N,\infty)$.
 \end{description}

Open problem 6c is actually not an open problem. The negative answer to this problem follows from the first part of Theorem \ref{ITmain}. Moreover, the answer is negative even if
we restrict ourselves to subspaces $\Tr(Q)$ of trigonometric polynomials.   The reader can find results in this paper, which give
partial progress in Open problems 6a and 6b. The most progress is made in case $q=2$. In case $q=2$ and $\mu$ is a discrete measure concentrated on $\Omega_M=\{x^j\}_{j=1}^M$ with $\mu(x^j)=1/M$, $j=1,\dots,M$, the answer to Open problem 6b is positive. This follows directly from (\ref{C2'}). It is clear that it can be generalized for many other probability measures $\mu$, for instance, for the Lebesgue measure on $\Omega$. It is likely that the answer to the Open problem 6b is "yes". Probably, the best progress in Open problem 6b for arbitrary $\mu$
is given in Theorem \ref{CT5'}. Certainly, the above two open problems are of interest in the case of trigonometric polynomials as well. We formulate them explicitly.

\begin{description}
  \item[Open problem 6at.]
 Let $\Omega:=\T^d$ and $q\in [1,\infty)$. Is the following statement true? There exists $C(d,q)$ such that for any $N$-dimensional
 subspace $X_N = \Tr(Q)$ we have $X_N\in \cM(C(d,q)N,q)$.
 \end{description}

 \begin{description}
  \item[Open problem 6bt.]
 Let $\Omega:=\T^d$ and $q\in [1,\infty)$. Is the following statement true? There exists $C(d,q)$ such that for any $N$-dimensional
 subspace $X_N = \Tr(Q)$ we have $X_N\in \cM^w(C(d,q)N,q)$.
 \end{description}

 Theorem \ref{NOUth} gives a positive answer to Open problems 6at and 6bt in the case $q=2$. In all other cases of $q$ we do not have an answer.

\begin{description}
  \item[Open problem 7.]
 In the case $q=\infty$ there is  the Kashin-Temlyakov phenomenon,
which says that for $\Tr(Q_n) \in \cM(m,\infty)$ it is necessary to have $m\ge c(d)|Q_n|^{1+c}$, $c>0$. Is it true that for all dimensions $\Tr(Q_n) \in \cM(m,\infty)$ provided $m\ge C(d)|Q_n|^2$?
\end{description}

The following is a weaker form of open problem 7.
\begin{description}
  \item[Open problem 8.] Theorem~\ref{thm-2-1} shows
that
$$
\Tr(Q_n) \in \cM(C_d2^{n\alpha_d}n^{\beta_d},\infty)
$$
with $\alpha_d \asymp \ln d$.
Does there exist an absolute constant $c$ such that
$$
\Tr(Q_n) \in \cM(C_d2^{cn},\infty) ?
$$
\end{description}



Assume that $X_N =\sp\{u_1(x),\dots,u_N(x)\}$ where
$\{u_i(x)\}_{i=1}^N$ is a real orthonormal  system on $\Omega$. The condition {\bf E} (see (\ref{ud5})) is a typical sufficient condition for some results.
For instance, let $\Omega_M=\{x^j\}_{j=1}^M$ be a discrete set with the probability measure $\mu(x^j)=1/M$, $j=1,\dots,M$.  Then it is known (Rudelson for $\Omega_M$, see \cite{VT159} for general $\Omega$) that
\be\label{3.3}
X_N \in \cM(CN\log N, 2).
\ee

It would be interesting to understand how important condition {\bf E} is for the Marcinkiewicz-type discretization theorems.

\subsection{$\cM^w(m,q)$, $1\le q\le \infty$}
\label{Mw}

For $q=2$ there is a strong result from \cite{BSS} (see a discussion in Subsection \ref{sec2.1} and at the end of Section 6 of \cite{VT159})
\be\label{3.4}
X_N(\Omega_M) \in \cM^w(m,2,\epsilon)\quad \text{provided} \quad m \ge CN\epsilon^{-2}
\ee
with large enough $C$.


\begin{description}
  \item[Open problem 9.]
  For which $X_N$ we have different conditions on $m$ for
$X_N\in\cM(m,q)$ and $X_N\in\cM^w(m,q)$?
\end{description}

\subsection{Constructive proofs}

Theorem \ref{gfT1} establishes the following inclusion for even positive integers $q$
$$
X_N \in \cM(M(N,q),q,0).
$$
The proof of Theorem \ref{gfT1} is not constructive. In Subsection \ref{sec2.4} we give a constructive proof of Theorem \ref{gfT1} in case $q=2$.

\begin{description}
  \item[Open problem 10.]Give a constructive proof of Theorem \ref{gfT1} for all even positive integers $q$.
  \end{description}

We pointed out in Section \ref{survey} that the main technique used for proving the Marcinkiewicz-type discretization theorems is a probabilistic technique.
\begin{description}
  \item[Open problem 11.] Give a constructive proof of Theorem \ref{NOUth}.
  \item[Open problem 12.]  Give a constructive proof of Theorem \ref{T6.1}.
  \item[Open problem 13.]  Give a constructive proof of Theorem \ref{T5.4}.
\end{description}

 \vskip 1.0cm


\Addresses
\end{document}